\documentclass[12pt]{amsart}
\usepackage{
   a4wide, 
  amsmath,
  amsfonts,
  amssymb,
  graphicx, 
  times,
  color,
  hyphenat, 
  stmaryrd, datetime, bbm, tikz-cd, thmtools, verbatim
}

\usepackage{dynkin-diagrams}
\usepackage{mathdots}
\usepackage[all]{xy}
\usepackage{euscript,mathrsfs}
\usepackage{bbm,xspace}
\usepackage[final]{hyperref, bookmark}
\usepackage{cleveref}
\hypersetup{colorlinks=true, pdfstartview=FitV, linkcolor=blue, citecolor=blue, urlcolor=blue}

\usepackage{tikz}
\usetikzlibrary{cd}
\usetikzlibrary{decorations}
\usetikzlibrary{decorations.markings}
\usetikzlibrary{decorations.pathreplacing}
\usetikzlibrary{decorations.pathmorphing}
\usetikzlibrary{arrows.meta,shapes,positioning,matrix,calc}
\usetikzlibrary{shapes.callouts}
\tikzstyle{densely dotted}=[dash pattern=on \pgflinewidth off 0.5pt]
\tikzset{anchorbase/.style={baseline={([yshift=-0.5ex]current bounding box.center)}},
tinynodes/.style={font=\tiny,text height=0.25ex,text depth=0.05ex},
smallnodes/.style={font=\scriptsize,text height=0.75ex,text depth=0.15ex},
usual/.style={line width=0.9,color=black},
dusual/.style={line width=0.9,color=spinach,densely dashed},
pole/.style={line width=3.0,color=specialgray},
crossline/.style={preaction={draw=white,line width=5.0pt,-},preaction={draw=black,line width=0.9pt,-}},
crosspole/.style={preaction={draw=white,line width=6.0pt,-},preaction={draw=specialgray,line width=3.0pt,-}},
mor/.style={line width=0.75,color=black,fill=cream},
blob/.style={circle,fill,minimum size=5.0pt,inner sep=0pt,outer sep=0pt},
blobed/.style n args={3}{decoration={markings,post length=0.5mm,pre length=0.5mm,
mark=at position #1 with {\node[blob,#3,label=left:$#2\!$]at (0,0){};}
},postaction={decorate}},
rblobed/.style n args={3}{decoration={markings,post length=0.5mm,pre length=0.5mm,
mark=at position #1 with {\node[blob,#3,label=right:$\!#2$]at (0,0){};}
},postaction={decorate}},
}

\newcommand{\tikzdiagh}[2][]{\tikz[#1,thick,baseline={([yshift=1ex+#2]current bounding box.center)}]}
\newcommand{\tikzdiagc}[1][]{\tikzdiagh[#1]{-1ex}}

\tikzstyle{tikzdot}=[fill, circle, inner sep=2pt]
\tikzstyle{smoltikzdot}=[fill=white, draw=black, circle, inner sep=1pt]

\tikzset{
    partial ellipse/.style args={#1:#2:#3}{
        insert path={+ (#1:#3) arc (#1:#2:#3)}
    }
}

\input xy
\xyoption{all}
\definecolor{myblue}{rgb}{0,.5,1}
\definecolor{myred}{rgb}{0.9,0,0}
\definecolor{mygreen}{rgb}{0,0.7,0}

\newcommand{\rB}{\mathrm{B}}
\newcommand{\rT}{\mathrm{T}}

\newcommand{\Bi}{\textcolor{blue}{\rB_i}}
\newcommand{\Bii}{\textcolor{myred}{\rB_{i-1}}}

\newcommand{\Bdd}{\textcolor{mygreen}{\rB_{d-1}}}
\newcommand{\Rrho}{\textcolor{violet}{\rT_{\rho}}}
\newcommand{\Rrhoi}{\textcolor{violet}{\rT_{\rho}^{-1}}}

\newcommand{\Ti}{\textcolor{blue}{\rT_i}}
\newcommand{\Tim}{\textcolor{blue}{\rT_i^{-1}}}
\newcommand{\Tipm}{\textcolor{blue}{\rT_i^{\pm 1}}}
\newcommand{\Timp}{\textcolor{blue}{\rT_i^{\mp 1}}}
\newcommand{\Tii}{\textcolor{myred}{\rT_{i-1}}}
\newcommand{\Tiim}{\textcolor{myred}{\rT_{i-1}^{-1}}}

\newcommand{\Timo}{\textcolor{myred}{\rT_{i+1}}}
\newcommand{\Timom}{\textcolor{myred}{\rT_{i+1}^{- 1}}}
\newcommand{\Tiipm}{\textcolor{myred}{\rT_{i+1}^{\pm 1}}}
\newcommand{\Tiimp}{\textcolor{myred}{\rT_{i+1}^{\mp 1}}}

\newcommand{\Saff}{\widehat{\eS}}
\newcommand{\Sext}{\widehat{\eS}^{\mathrm{ext}}}
\newcommand{\BSext}{\widehat{\eBS}^{\mathrm{ext}}}

\DeclareMathOperator{\ev}{ev}
\newcommand{\Ev}{\mathcal{E}v}

\newcommand{\vastl}[1]{\left(\rule{#1 cm}{.9cm}\right.}
\newcommand{\vastr}[1]{\left.\rule{#1 cm}{.9cm}\right)}

\newcommand{\affh}{\widehat{H}_d}
\newcommand{\eaffh}{\widehat{H}^{ext}_d}
\newcommand{\Sy}{\mathfrak{S}}

\newcommand{\kb}{{\eK^b}}

\newtheorem{thm}{Theorem}[section]
\newtheorem{lem}[thm]{Lemma}
\newtheorem{cor}[thm]{Corollary}
\newtheorem{prop}[thm]{Proposition}

\theoremstyle{definition}
\newtheorem{defn}[thm]{Definition}

\newtheorem{rem}[thm]{Remark}

\newcommand{\bN}{\mathbb{N}}

\newcommand{\bC}{\mathbb{C}}

\newcommand{\eA}{\EuScript{A}}

\newcommand{\eC}{\EuScript{C}}

\newcommand{\eK}{\EuScript{K}}

\newcommand{\eS}{\EuScript{S}}

\newcommand{\eBS}{\EuScript{BS}}

\newcommand{\cC}{\mathcal{C}}

\DeclareMathOperator{\id}{Id}

\newcommand{\xra}[1]{\xrightarrow{#1}}

\catcode`\@=11
\long\def\@makecaption#1#2{%
    \vskip 10pt
    \setbox\@tempboxa\hbox{%
\small{#1: }\ignorespaces #2}%
    \ifdim \wd\@tempboxa >\captionwidth {%
        \rightskip=\@captionmargin\leftskip=\@captionmargin
        \unhbox\@tempboxa\par}%
      \else
        \hbox to\hsize{\hfil\box\@tempboxa\hfil}%
    \fi}
\newdimen\@captionmargin\@captionmargin=2\parindent
\newdimen\captionwidth\captionwidth=\hsize
\catcode`\@=12

\let\fullref\autoref
\def\makeautorefname#1#2{\expandafter\def\csname#1autorefname\endcsname{#2}}
\makeautorefname{lem}{Lemma}%
\makeautorefname{prop}{Proposition}%
\makeautorefname{rem}{Remark}%
\makeautorefname{section}{Section}%
\makeautorefname{subsection}{Section}%
\makeautorefname{subsubsection}{Section}%

\title{Evaluation birepresentations of affine type $A$ Soergel bimodules}

\author{Marco Mackaay}
\address{M.M.: Center for Mathematical Analysis, Geometry, and Dynamical Systems, Departamento de Matemática, Instituto Superior Técnico, 1049-001 Lisboa, PORTUGAL \&
Departamento de Matemática, FCT, Universidade do Algarve, Campus de Gambelas,
8005-139 Faro, PORTUGAL, \newline \href{https://fct.ualg.pt/bio/mmackaay}{https://fct.ualg.pt/bio/mmackaay}, \href{https://orcid.org/0000-0001-9807-6991}{ORCID 0000-0001-9807-6991}}
\email{mmackaay@ualg.pt}
\author{Vanessa Miemietz}
\address{V.M.: School of Mathematics, University of East Anglia, Norwich NR4 7TJ, United Kingdom,  \newline \href{https://archive.uea.ac.uk/~byr09xgu/}{https://archive.uea.ac.uk/~byr09xgu/}}
\email{v.miemietz@uea.ac.uk}
\author{Pedro Vaz}
\address{P.V.: Institut de Recherche en Math{\'e}matique et Physique, 
Universit{\'e} catholique de Louvain, Chemin du Cyclotron 2,  
1348 Louvain-la-Neuve, Belgium, \newline \href{https://perso.uclouvain.be/pedro.vaz}{https://perso.uclouvain.be/pedro.vaz}, \href{https://orcid.org/0000-0001-9422-4707}{ORCID 0000-0001-9422-4707}}
\email{pedro.vaz@uclouvain.be}

\begin{document}
%
%
%
\begin{abstract} In this paper, we use Soergel calculus to define a monoidal functor, called the evaluation functor, 
from extended affine type $A$ Soergel bimodules to the homotopy category of bounded complexes 
in finite type $A$ Soergel bimodules. This functor categorifies the well-known evaluation homomorphism 
from the extended affine type $A$ Hecke algebra to the finite type $A$ Hecke algebra. 
Through it, one can pull back the triangulated birepresentation induced by any finitary birepresentation of finite type 
$A$ Soergel bimodules to obtain a triangulated birepresentation of extended affine type $A$ Soergel bimodules. 
We show that if the initial finitary birepresentation in finite type $A$ is a cell birepresentation, 
the evaluation birepresentation in extended affine type $A$ has a finitary cover, which we illustrate by working out 
the case of cell birepresentations with subregular apex in detail.     
\end{abstract}
\maketitle

{\hypersetup{hidelinks}
\tableofcontents 
}
%
%
\pagestyle{myheadings}
\markboth{\em\small M.~Mackaay, V.~Miemietz, P.~Vaz}{\em\small  Evaluation birepresentations}
%
%

\section{Introduction}

Finitary birepresentation theory of finite type Soergel bimodules in characteristic zero 
has been a topic of intensive study, with many interesting results, in the last couple 
of years~\cite{kmmz,mackaay-mazorchuk,mmmtz2019,mackaay-tubbenhauer,zimmermann}. 
In this paper, we initiate the study of a class of finitary and triangulated birepresentations 
of affine type $A$ Soergel bimodules. The bicategories of these Soergel bimodules are no longer finitary and, therefore, new phenomena show up in their birepresentation theory. For example, there are no known interesting triangulated birepresentations in finite type, whereas we do give examples of such birepresentations in affine type $A$. 

To describe these, let us briefly recall the decategorified setting first. In type $A$, 
as is well-known, there are evaluation maps from the affine Hecke algebra 
to the finite type Hecke algebra. These are homomorphisms of algebras, 
so any representation of the latter algebra can be pulled back to a representation 
of the former algebra through such a map. These so-called {\em evaluation representations} 
form an important and well-studied class of finite-dimensional representations of affine type $A$ Hecke algebras, see e.g.~\cite{chari-pressley, dufu, lnt} and references therein.  
 
Several authors (\cite[Introduction]{mackaay-thiel} and \cite[Section 1.6]{elias2018}) 
have conjectured that these evaluation maps can be categorified by 
monoidal {\em evaluation functors} (i.e., pseudofunctors between one-object bicategories) from affine type $A$ Soergel bimodules to  the homotopy category of bounded complexes in finite type $A$ Soergel bimodules. In this paper, we indeed define such functors and use them to categorify the aforementioned 
evaluation representations in the form of triangulated birepresentations, obtained by 
pulling back the triangulated birepresentations induced by finitary birepresentations 
of finite type $A$ Soergel bimodules through these functors. Moreover, in case 
the original finitary birepresentation is simple transitive, we show that the evaluation 
birepresentation admits a {\em finitary cover}, i.e., a finitary birepresentation 
together with an essentially surjective and epimorphic morphism of 
additive birepresentations from that cover to the evaluation birepresentation. 
This categorifies the well-known fact that the corresponding 
evaluation representations are quotients of certain 
cell representations defined by Graham and Lehrer~\cite{Graham-Lehrer}.  

Let us finish this introduction with a disclaimer. We do not present a theory of triangulated birepresentations in this paper. First of all, it is not yet clear 
whether our evaluation functors can be extended to triangulated functors between the homotopy category of bounded complexes in affine type $A$ Soergel bimodules and its counterpart in finite type $A$. Proving the existence of such an extension is a non-trivial exercise in obstruction theory, which will have to be addressed in the future. This extension 
problem was first mentioned in \cite[Section 16]{elias2018}, where it is conjectured to be solvable, and a similar problem will have to be solved in order to prove~\cite[Conjecture 1.2]{AL-ELR} for the categorification of the internal braid group action on quantum groups.  Secondly, some ingredients for a theory of 
triangulated birepresentations can already be found in the literature, e.g.~\cite{elias2018, elias-hogancamp, hogancamp, laugwitz-miemietz, stevenson}, but many foundational results are still missing. In general, it is not clear which parts of finitary birepresentation theory, e.g. the notion of simple transitive birepresentation, the categorical (weak) Jordan-H\"older theorem, the relation with (co)algebra $1$-morphisms, the double-centralizer theorem (see~\cite{mmmtz2020} and references therein), 
generalize to the triangulated setting and/or in which form exactly. These questions need 
to be answered first, before one can even think of categorifying the induction product 
of evaluation representations from~\cite[Section 2.5]{lnt}. Finally, all of this is just 
for affine type $A$. Hecke algebras of other affine Coxeter types also have interesting 
finite-dimensional representations, but there are no evaluation morphisms in those cases, so other ideas will be needed to categorify those representations. In other words, the results in this paper are (hopefully) just the tip of a (tricky) triangulated iceberg.

\subsubsection*{Plan of the paper} 

In Section~\ref{sec:decatreminders}, we recall the basics of extended and 
non-extended affine Hecke algebras of affine type $A$, the evaluation maps, 
the Graham-Lehrer cell modules and the evaluation representations. Everything in this section 
is well-documented in the literature and we only recall the details that are needed in the 
rest of this paper.    

In Section~\ref{sec:soergel}, we briefly recall Soergel calculus in finite and 
affine type $A$, the latter both in the non-extended and the extended version. 
Again, nothing new is presented, so the specialists can skip this section and 
move on to the next one. Of course, in the remainder 
we often refer to the diagrammatic equations in this section, which is exactly why 
we recall them. 
   
In Section~\ref{sec:Rouquier}, we first recall some basic results on Rouquier complexes in 
finite type $A$ and then focus on a special type of Rouquier complex, 
which is fundamental for the definition of the evaluation functors in the next section. 
In particular, we develop a mixed diagrammatic calculus for morphisms between products of Bott-Samelson bimodules and these special Rouquier complexes, all in finite type $A$. 
To the best of our knowledge, this extension of the usual Soergel calculus is new.  
 
In Section~\ref{sec:evaluationfunctors}, we define the evaluation functors by 
assigning a bounded complex of finite type $A$ Soergel bimodules (or, more precisely, of 
finite type $A$ Bott-Samelson bimodules) to each extended affine type 
$A$ Bott-Samelson bimodule and a map 
between such complexes to each generating extended affine type $A$ 
Soergel calculus diagram. The main result of this section, and of this paper, 
is that this assigment is well-defined up to homotopy equivalence. 

In Section~\ref{sec:bireps}, we first introduce the notion of a triangulated birepresentation of an 
additive bicategory and define evaluation birepresentations of Soergel bimodules in extended 
affine type $A$, which are important examples. We then prove that 
each evaluation birepresentation has a (possibly non-unique)
finitary cover. Finally, we study in detail the simplest non-trivial evaluation birepresentations, 
which are the ones induced by cell birepresentations of finite type $A$ with subregular apex. 
As we show, these admit a simple transitive finitary cover whose underlying algebra 
is a signed version of the zigzag algebra of affine type $A$.

\subsubsection*{Acknowledgments} 
We thank the anonymous referee for carefully reading our paper and giving us for some very helpful comments. 

M.M. was supported in part by Funda\c{c}{\~a}o para a Ci\^{e}ncia e a
Tecnologia (Portugal), projects UID/MAT/04459/2013 (Center for Mathematical Analysis,
Geometry and Dynamical Systems - CAMGSD) and PTDC/MAT-PUR/31089/2017 (Higher Structures and Applications). V.M. was partially supported by EPSRC grant EP/S017216/1.
P.V. was supported by the Fonds de la Recherche Scientifique - FNRS under Grant no. MIS-F.4536.19.


%
%
\section{The decategorified story}\label{sec:decatreminders}

From now on, fix $d\in\bN_{\geq 3}$ and let 
$\widehat{I}:=\mathbb{Z}/d\mathbb{Z}$ and $I:=\{1,\ldots, d-1\}$. By a slight abuse of notation, we will often identify $\widehat{I}$ with the set of representatives 
$\{0,1,\ldots, d-1\}$ and consider $I$ as a subset of $\widehat{I}$. 

Let $\widehat{\Sy}_d$ be the affine Weyl group of type $\widehat{A}_{d-1}$. It is generated by $s_i, \; i\in \widehat{I},$ subject to relations
\begin{equation*}
  s_i^2 = 1, \mspace{40mu} s_is_j =s_js_i \mspace{10mu}\text{if}\mspace{10mu}
  \vert i-j\vert>1, \mspace{40mu} s_is_{i+1}s_i = s_{i+1}s_is_{i+1}, 
\end{equation*}
for $i\in \widehat{I}$. The {\em extended} affine Weyl group $\widehat{\Sy}_d^{\mathrm{ext}}$ is the semidirect product 
\[
\langle \rho \rangle \ltimes \widehat{\Sy}_d,
\]
where $\langle \rho \rangle$ is an infinite cyclic group generated by $\rho$ and 
\[
\rho s_i \rho^{-1} =s_{i+1},
\]
for $i\in \widehat{I}$. The finite Weyl group of type $A_{d-1}$ is the symmetric group on $d$ letters, $\Sy_d$, corresponding to the subgroup of $\widehat{\Sy}_d$ generated by $s_i,\; i\in I$.

\begin{rem} In some papers, the name {\em extended affine Weyl group} of type 
$\widehat{A}_{d-1}$ is used for the quotient of  $\widehat{\Sy}_d$ by the ideal 
generated by $\rho^d$. However, there are no evaluation maps from 
the extended affine Hecke algebra corresponding to that quotient to the finite type Hecke algebra, so we will not consider it in this paper.   
\end{rem}

\subsection{Hecke algebras}\label{sec:hecke}

Let $\Bbbk=\bC(q)$, where $q$ is a formal parameter. The \emph{extended affine Hecke} algebra $\eaffh$ is the $\Bbbk$-algebra generated by
$T_i,\; i\in \widehat{I},$ and $\rho^{\pm 1}$, with relations
\begin{gather}
  (T_i+q)(T_i-q^{-1}) = 0, \mspace{40mu} T_iT_j =T_jT_i \mspace{10mu}\text{if}\mspace{10mu}
  \vert i-j\vert>1, \mspace{40mu} T_iT_{i+1}T_i = T_{i+1}T_iT_{i+1}, 
 \label{eq:affHeckeSnrels}   \\
\rho\rho^{-1}=1=\rho^{-1}\rho,\mspace{40mu} \rho T_i \rho^{-1} = T_{i+1},  
    \label{eq:affHeckeRrels} 
    \end{gather}
for $i,j\in \widehat{I}$. Note that $T_i$ is invertible 
for every $i\in \widehat{I}$ with
\[
T_i^{-1}=T_i+q-q^{-1}.
\] 
As is well-known, $\eaffh$ is a $q$-deformation of the group 
algebra $\mathbb{C}[\widehat{\Sy}_d^{\mathrm{ext}}]$ with basis (the {\em regular basis})  given by 
$\{\rho^m T_w\mid m\in \mathbb{Z}, w\in \widehat{\Sy}_d\}$, where $T_w:=T_{i_1} \cdots T_{i_{\ell}}$ 
for any {\em reduced expression} (rex) $s_{i_1}\cdots s_{i_{\ell}}$ of $w$.

\smallskip

Another presentation is given in terms of the \emph{Kazhdan--Lusztig generators} $b_i:=T_i+q$, for $i\in \widehat{I}$, 
and $\rho^{\pm 1}$, subject to relations
\begin{gather}
  b_i^{2}= [2] b_i, \mspace{40mu} b_ib_j =b_jb_i \mspace{10mu}\text{if}\mspace{10mu}
  \vert i-j\vert>1, \mspace{40mu} b_ib_{i+1}b_i + b_{i+1} = b_{i+1}b_ib_{i+1} + b_{i}, 
 \label{eq:fdHeckebrels}   \\
\rho\rho^{-1}=1=\rho^{-1}\rho,\mspace{40mu} \rho b_i \rho^{-1} = b_{i+1}, 
    \label{eq:affHeckebrels} 
    \end{gather}
for $i\in \widehat{I}$, where $[2]:=q+q^{-1}$. Note that $T_i=b_i-q$ and $T_i^{-1}=b_i-q^{-1}$, for every 
$i\in \widehat{I}$. The {\em Kazhdan--Lusztig basis} is given by $\{\rho^m b_w \mid 
m\in \mathbb{Z}, w \in \widehat{\Sy}_d \}$, where $b_w$ is defined for an 
arbitrary rex of $w$ (and is independent of that choice). 

\smallskip 

The (non-extended) \emph{affine Hecke algebra} $\affh$ is the subalgebra of $\eaffh$ generated by either $T_i,\; i\in \widehat{I},$ subject to relations~\eqref{eq:affHeckeSnrels},
or $b_i,\; i\in \widehat{I},$ subject to relations~\eqref{eq:fdHeckebrels}. 

\smallskip

The \emph{finite Hecke algebra} $H_d$ is the $\Bbbk$-subalgebra of $\affh$ generated by 
either $T_i,\; i\in I,$ subject to relations~\eqref{eq:affHeckeSnrels}, or  
$b_i,\; i\in I\,$ subject to relations~\eqref{eq:fdHeckebrels}.

\subsection{Evaluation maps}
\begin{defn}\label{defn:evaluationmap}
For any $a\in\mathbbm{k}^{\times}$, there are two \emph{evaluation maps} 
$\ev_a, \ev'_a\colon\eaffh\to H_d $. These are defined as the homomorphisms of 
$\mathbbm{k}$-algebras determined by 
\begin{align}
  \ev_a(T_i) &= T_i,   \quad \mathrm{for}\;i\in I,
  \\
  \ev_a(\rho) &= a T_1^{-1}\dotsm T_{d-1}^{-1} 
\end{align}
and 
\begin{align}
  \ev'_a(T_i) &= T_i,   \quad \mathrm{for}\;i \in I,
  \\
  \ev'_a(\rho) &= a T_1\dotsm T_{d-1},
\end{align}
respectively. 
\end{defn}
The definition implies that 
\begin{equation}
\ev_a(T_0) =  \ev_a(\rho^{-1}T_1\rho) = T_{d-1}\dotsm T_2T_1 T_2^{-1} \dotsm 
T_{d-1}^{-1} 
\end{equation}
and 
\begin{equation}
\ev'_a(T_0) =  \ev'_a(\rho^{-1}T_1\rho) = T^{-1}_{d-1}\dotsm T^{-1}_2T_1 
T_2 \dotsm T_{d-1}, 
\end{equation}
so the restrictions of $\ev_a$ and $\ev'_a$ to $\affh$ do not depend on $a$. 

In terms of the Kazhdan--Lusztig generators we have
\begin{align}
  \ev_a(b_i)  &= b_i,\quad \mathrm{for}\; i \in I,\\
  \ev_a(b_0) &=  \ev_a(\rho^{-1}b_1\rho) = 
(b_{d-1}-q)\dotsm(b_{1}-q)b_1(b_1-q^{-1})\dotsm (b_{d-1}-q^{-1})  
\end{align}
and 
\begin{align}
  \ev'_a(b_i)  &= b_i,\quad \mathrm{for}\; i \in I,\\
  \ev'_a(b_0) &=  \ev'_a(\rho^{-1}b_1\rho) = 
(b_{d-1}-q^{-1})\dotsm(b_{1}-q^{-1})b_1(b_1-q)\dotsm (b_{d-1}-q).  
\end{align}
Another way of saying this is that the evaluation maps do not preserve the bar involution,  
but rather satisfy 
\begin{equation}\label{eq:rel-evaluation-maps}
\overline{\ev_a(x)}=\ev'_{\overline{a}}(\overline{x}),
\end{equation}
for any $x\in \eaffh$ and $a=a(q)\in \mathbbm{k}^{\times}$.

One can also define $\ev_a$ and $\ev'_a$ using a third presentation of $\eaffh$, called the \emph{Bernstein presentation}. In that presentation, $\eaffh$ is defined as some sort of semidirect product of $H_d$ and $\mathbbm{k}[Y_1^{\pm 1}\dotsc ,Y_d^{\pm 1}]$. However, there are several possible choices for the algebra of 
Laurent polynomials. In~\cite{elias2018}, 
two such choices are given with different variables: $y_1,\dotsc , y_d$ and 
$y_1^*,\dotsc ,y_d^*$ respectively. The interaction of $H_d$ and these polynomial algebras 
is defined by
\begin{align}
T_i^{-1}y_iT_i^{-1} &= y_{i+1}
\intertext{and}
T_i y_i^* T_i &= y_{i+1}^*, 
\end{align}  
respectively, for $i\in I$. 

The relation between these two Bernstein presentations and our first presentation of $\eaffh$ is 
given by 
\begin{align}
y_1 &= \rho T_{d-1}\dotsm T_{2}T_1 , \\
y_i &= T_{i-1}^{-1}\dotsm T_2^{-1}T_1^{-1}\rho T_{d-1}\dotsm T_{i+1}T_i, \quad i=2,
\ldots, d-1, 
\intertext{resp.}
y_1^* &= \rho T_{d-1}^{-1}\dotsm T_{2}^{-1}T_1^{-1}  , \\
y_i^* &= T_{i-1}\dotsm T_2T_1\rho T_{d-1}^{-1}\dotsm T_{i+1}^{-1}T_i^{-1}, \quad 
i=2,\ldots,d-1. 
\end{align}

It follows that the evaluation map $\ev_a\colon \eaffh\to H_d$ is the unique homomorphism of algebras sending $T_i$ to $T_i$, for $i\in I$, and $y_1$ to $a$, while $\ev'_a\colon \eaffh\to H_d$ is the unique homomorphism of algebras 
sending $T_i$ to $T_i$, for $i\in I$, and $y_1^*$ to $a$. The latter 
coincides with the flattening map $\flat$ in~\cite[\S 2.6]{elias2018} for $a=1$. 

We will categorify the evaluation map $\ev_a$ in~\fullref{sec:defevalfunctor}. The 
categorification of $\ev'_a$ is very similar and the relation between the two 
evaluation maps in~\eqref{eq:rel-evaluation-maps} also categorifies, 
since the categorification of the bar-involution is given by flipping diagrams upside-down, 
inverting the orientation of the differentials in complexes and changing the sign of homological 
and grading shifts. 

\begin{rem}
Some remarks about the various conventions in the literature are in order.
We try to follow conventions close to those in~\cite{elias2018}. 
Our presentation of the extended affine Hecke algebra in~\fullref{sec:hecke}
agrees with~\cite{elias2018}, as does the relation between the standard generators and the 
Kazhdan--Lusztig generators. Some authors use the inverse of $\rho$ in~\eqref{eq:affHeckeRrels}. 
Our choice of conventions implies the absence of certain powers of $q$ in the definition of the evaluation maps, in comparison with some of the sources in the literature. 
For more information on evaluation maps, see e.g.~\cite[\S5.1]{chari-pressley} and \cite[(5.0.2)]{dufu}. There are more possible evaluation maps, but we only consider these two in this paper. 
\end{rem}


\subsection{Graham-Lehrer cell modules}
Consider the $\widehat{A}_{d-1}$ Coxeter diagram $\widehat{\Gamma}_{d-1}$ with its vertices ordered counterclockwise and top vertex numbered $0$, e.g.    
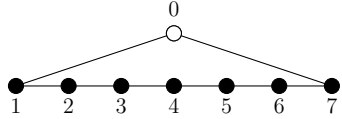
\[
\dynkin [ordering=Kac,label,scale=2] A[1]7
\]
for $d=8$.
For any $z\in \mathbbm{k}^{\times}$, the {\em Graham-Lehrer cell module} 
$\widehat{M}_z$ of $\affh$ corresponding to $z$ and the partition $(d-1,1)$ 
has underlying vector space 
\begin{equation}
\widehat{M}_z:=\mathrm{Span}_{\mathbbm{k}}\left\{m_i \mid i\in \widehat{I}\right\} 
\end{equation}
and the action of $\affh$ 
on $\widehat{M}_z$ is given by  
\begin{equation}\label{eq:GL}
b_i m_j = 
\begin{cases}
[2] m_i, & \text{if}\; j\equiv i \bmod d;\\
z m_1, & \text{if}\; i-1\equiv 0\equiv j \bmod d;\\
z^{-1} m_0, & \text{if}\; i\equiv 0\equiv j-1 \bmod d;\\
m_j, &\text{if}\; i\equiv j\pm 1\bmod d,\;\text{but none of the above};\\
0, & \text{else}.
\end{cases}
\end{equation}

It is easy to see that $\widehat{M}_z$ is isomorphic to $W_{d-2,\pm \sqrt{z}}(d)$ 
in~\cite[Definition 2.6]{Graham-Lehrer}, where $m_i$ is identified with the cup 
diagram on a cylinder with $d-2$ straight lines and only one cup, whose endpoints are $i$ and $i+1$. When $i\ne 0$, the whole diagram corresponding to $m_i$ lives on the front part of the cylinder, but when $i=0$, the cup of $m_0$ goes around the back of the cylinder. Note that we have used $\delta=[2]$, rather than $\delta=-[2]$. As remarked in~\cite[text above Corollary 2.9.1]{Graham-Lehrer}, 
$W_{d-2,\sqrt{z}}(d)$ and $W_{d-2, -\sqrt{z}}(d)$ are isomorphic, which is clear 
from the fact that both are isomorphic to $\widehat{M}_z$. 

The Graham-Lehrer cell module $\widehat{M}_z$ can be made into an $\eaffh$-module, 
but not in a unique way. As a matter of fact, for each $\lambda\in \mathbbm{k}^{\times}$, 
we can define    
\begin{equation}\label{eq:GLext}
\rho\, m_j = \lambda z^{\delta_{j,0}}m_{j+1}, 
\end{equation}
for $j\in \widehat{I}$. It is easy to verify that this gives a well-defined action and 
we denote the corresponding Graham-Lehrer cell module of $\eaffh$ 
by $\widehat{M}_{z,\lambda}$. Note that the restriction of $\widehat{M}_{z,\lambda}$ to 
$\affh$ is equal to $\widehat{M}_z$, for all $\lambda\in \mathbbm{k}^{\times}$, and 
that the action of $\rho^{d}$ on $\widehat{M}_{z,\lambda}$ 
is simply multiplication by $\lambda^d z$. 

Graham and Lehrer~\cite[Theorem 2.8]{Graham-Lehrer} defined a 
$\mathbbm{k}$-bilinear form 
\begin{equation}
\langle \cdot , \cdot \rangle \colon \widehat{M}_z \otimes \widehat{M}_{z^{-1}}
\to \mathbbm{k},
\end{equation}
which in our notation is determined by    
\begin{equation}
\langle m_i, m_j\rangle =
\begin{cases}
[2], & \text{if}\; j \equiv i \bmod d;\\
z, & \text{if}\; i\equiv 0\equiv j-1 \bmod d;\\
z^{-1}, & \text{if}\; i-1\equiv 0\equiv j \bmod d;\\
1, &\text{if}\; i\equiv j\pm 1\bmod d,\;\text{but none of the above};\\
0, & \text{else}.
\end{cases}
\end{equation}
This induces a $\mathbbm{k}$-bilinear form on $\widehat{M}_{z,\lambda}\otimes 
\widehat{M}_{z^{-1},\lambda^{-1}}$, satisfying 
$\langle \rho^n b_w m_j, m_k\rangle = \langle m_j, b_w^{\star}\rho^{-n} m_k\rangle$, 
for any $w\in \widehat{W}$, $n\in \mathbb{Z}$ and $j,k\in \widehat{I}$, 
where $b_w^{\star}=b_{w^{-1}}$ is the dual Kazhdan-Lusztig basis element. Therefore, 
the radical of the bilinear form
\[
\mathrm{rad}(\langle \cdot, \cdot\rangle)=
\left\{m\in \widehat{M}_{z,\lambda}\mid \langle m,m'\rangle =0,\; \forall  
m'\in \widehat{M}_{z^{-1}, \lambda^{-1}} \right\} 
\]
is an $\eaffh$-submodule of $\widehat{M}_{z,\lambda}$. 
Graham and Lehrer~\cite[Theorem 2.8]{Graham-Lehrer} proved that the quotient module 
$\widehat{M}_{z}/\mathrm{rad}(\langle \cdot , \cdot \rangle)$ of $\affh$ is simple, 
and the same holds for the quotient module 
$\widehat{M}_{z,\lambda}/\mathrm{rad}(\langle \cdot , \cdot \rangle)$ of $\eaffh$, of course. 
A straightforward calculation shows that the radical of the bilinear form on 
$\widehat{M}_{z,\lambda}$ is zero unless $z=(-q)^{\pm d}$ (independently of $\lambda$), in which case it has dimension one and is generated by 
\begin{equation}
n_{\pm}:=\sum_{k=1}^{d} (-q)^{\mp k} m_k.
\end{equation}
Note that, when $z=(-q)^{\pm d}$, we have 
$\rho\, n_{\pm} =\lambda (-q)^{\pm 1} n_{\pm} $ and $b_i n_{\pm }=0$ for all $i\in \widehat{I}$. 

When $z=(-q)^{\pm d}$, put $\widehat{M}_{d, \lambda}^{\pm}:=
\widehat{M}_{(-q)^{\pm d}, \lambda^{\pm 1}}$ and let   
\begin{equation}\label{eq:Lplusminus}
\widehat{L}_{d, \lambda}^{\pm} :=\widehat{M}_{d, \lambda}^{\pm} /\langle n_{\pm} \rangle
\end{equation}
be the simple quotient $\eaffh$-modules of dimension $d-1$. Finally, denote the restriction 
of these simple modules to $\affh$ by 
\begin{equation}\label{eq:Lplusminusnonext}
\widehat{L}_{d}^{\pm} :=\widehat{M}_{d}^{\pm} /\langle n_{\pm} \rangle.  
\end{equation}
As explained above, these restrictions do not depend on $\lambda\in \mathbbm{k}^{\times}$. 

\subsection{Evaluation modules}
Let $M$ be a finite-dimensional $H_d$-module (over $\mathbbm{k}$). Recall that, 
for any $a\in \mathbbm{k}^{\times}$, there are two evaluation maps 
$\ev_a, \ev'_a\colon \eaffh\to H_d$ (see Definition~\ref{defn:evaluationmap}). 
\begin{defn}\label{defn:evaluationmodule}
For any $a\in \mathbbm{k}^{\times}$, the {\em evaluation modules} $M^{\mathrm{ev}_a}$ 
and $M^{\mathrm{ev}'_a}$ of $\eaffh$ are the pull-backs of $M$ through $\ev_a$ and $\ev'_a$, 
respectively. 
\end{defn}
The actions of $\eaffh$ on $M^{\ev_a}$ and $M^{\ev'_a}$ can be computed using the explicit formulas in Definition~\ref{defn:evaluationmap} and below. In this paper, we only consider the case when $M:=M_d$ is the simple 
$H_d$-module corresponding to the partition $(d-1,1)$. There are several ways to define $M_d$ explicitly and the definition we choose here is tailor-made for categorification. Take $M_d:=\mathrm{span}_{\mathbbm{k}}\{m_i\mid i\in I\}$, with the action of $H_d$ being given by  
\begin{equation}\label{eq:M}
b_i m_j = 
\begin{cases}
[2] m_i, & \text{if}\; j=i;\\
m_i, & \text{if}\; j=i\pm 1;\\
0, & \text{else},
\end{cases} 
\end{equation}
for $i,j\in I$. It is easy to show that $M_d$ is simple, but this is well-known so we leave it as an exercise to the reader. The action of the 
$T_i^{\pm 1}=b_i-q^{\pm 1}$ is also easy to give explicitly:
\begin{equation}
T_i^{\pm 1} m_j= 
\begin{cases}
q^{\mp 1} m_i, & \text{if}\; j=i;\\
m_i - q^{\pm 1} m_j, & \text{if}\; j=i\pm 1;\\
-q^{\pm 1} m_j, & \text{else}.
\end{cases}
\end{equation} 

Note that, as vector spaces, $M^{\ev_a}_d=M^{\ev'_a}_d=M_d$, and the action of 
$b_i\in \eaffh$, for $i\in I$, is the same as above because $\ev_a(b_i)=b_i$. 
A simple calculation now shows that  
\begin{equation}\label{eq:action-rho}
\ev_a(\rho) m_j= aT_1^{-1} \cdots T_{d-1}^{-1} m_j = 
\begin{cases}
a(-q)^{2-d} m_{j+1}, & \text{if}\; j=1,\ldots, d-2;\\
aq \sum_{k=1}^{d-1} (-q)^{1-k} m_{k},  & \text{if}\; j=d-1,
\end{cases}
\end{equation}
and 
\begin{equation}\label{eq:action-prime-rho}
\ev'_a(\rho) m_j= aT_1 \cdots T_{d-1} m_j = 
\begin{cases}
a(-q)^{d-2} m_{j+1}, & \text{if}\; j=1,\ldots, d-2;\\
aq^{-1} \sum_{k=1}^{d-1} (-q)^{k-1} m_{k},  & \text{if}\; j=d-1.
\end{cases}
\end{equation}
The actions of $b_0$ can then be computed using the equation $b_0=\rho^{-1} b_1 \rho$, 
but we omit the calculation because will not need the result.

Recall the simple quotients $\widehat{L}_{d, \lambda}^{\pm}$ of the Graham-Lehrer 
cell modules $\widehat{M}_{d, \lambda}^{\pm}$,  
defined in~\eqref{eq:Lplusminus}. 
\begin{thm}\label{thm:isosandpairing}
Let $a=\lambda (-q)^{d-2}$. There are two isomorphisms of $\eaffh$-modules 
\begin{eqnarray*}
\widehat{L}_{d, \lambda}^{+}&\cong & M^{\mathrm{ev}_a}_d; \\
 \widehat{L}_{d, \lambda}^{-}&\cong & M^{\mathrm{ev}'_{a^{-1}}}_d.
\end{eqnarray*}
Moreover, there is a perfect pairing of $\eaffh$-modules
\[
M_d^{\ev_a}\otimes M_d^{\ev'_{a^{-1}}}\to \mathbbm{k}. 
\] 
\end{thm}
\begin{proof}
To show the first part, it suffices to compute the action of $\rho$ on $\widehat{L}_{d, \lambda}^{+}$ and compare it to~\eqref{eq:action-rho}. 
Let $\overline{m}_k$ be the image of $m_k$ under 
the projection $\widehat{M}_{d, \lambda}^{+}\to \widehat{L}_{d, \lambda}^{+}$, 
for $k\in \widehat{I}$. Then 
$\left\{\overline{m}_1,\ldots, \overline{m}_{d-1}\right\}$ is a basis of 
$\widehat{L}_{d, \lambda}^{+}$, because $\overline{m}_0= -\sum_{k=1}^{d-1} (-q)^{k} \overline{m}_{d-k}$. This implies that in $\widehat{L}_{d, \lambda}^{+}$ 
we have 
\[
\rho\, \overline{m}_j=
\begin{cases}
\lambda \overline{m}_{j+1}, & \text{if}\, j=1,\ldots,d-2; \\
- \lambda \sum_{k=1}^{d-1} (-q)^{k} \overline{m}_{d-k},& \text{if}\, j=d-1.
\end{cases} 
\]
This is indeed the same as in \eqref{eq:action-rho} because $aq=\lambda (-q)^{d-2}q
=-\lambda (-q)^{d-1}$.  

Similarly, $\widehat{L}_{d, \lambda}^{-}\cong M^{\mathrm{ev}'_{a^{-1}}}_d$, 
as in $\widehat{L}_{d,\lambda}^{-}$ we have 
$\overline{m}_0= -\sum_{k=1}^{d-1} (-q)^{-k} \overline{m}_{d-k}$, so 
\[
\rho\, \overline{m}_j=
\begin{cases}
\lambda^{-1} \overline{m}_{j+1}, & \text{if}\, j=1,\ldots,d-2; \\
- \lambda^{-1} \sum_{k=1}^{d-1} (-q)^{-k} \overline{m}_{d-k},& \text{if}\, j=d-1,
\end{cases} 
\]
which is the same as in \eqref{eq:action-prime-rho} because $a^{-1}q^{-1}=
\lambda^{-1} (-q)^{2-d}q^{-1}=-\lambda^{-1}(-q)^{1-d}$.

For the second part, note that the two $\eaffh$-modules $\widehat{L}_{d, \lambda}^{+}$ and 
$\widehat{L}_{d, \lambda}^{-}$ are {\em dual} to each other, because we could 
also consider the radical defined by
\[
\mathrm{rad}'(\langle \cdot, \cdot\rangle)=
\left\{m'\in \widehat{M}_{z^{-1}, \lambda^{-1}}\mid \langle m,m'\rangle =0,\; \forall  
m\in \widehat{M}_{z, \lambda} \right\}, 
\]
which is an $\eaffh$-submodule of $\widehat{M}_{z^{-1}, \lambda^{-1}}$. 
As before, this radical is zero unless $z=(-q)^{\pm d}$. For these two values of $z$ and any 
value of $\lambda\in \mathbbm{k}^{\times}$, 
the two simple quotients of $\widehat{M}_{z^{-1},\lambda^{-1}}$ 
are isomorphic to $\widehat{L}_{d,\lambda}^{\mp}$ and 
the bilinear form descends to a perfect pairing 
\[
\widehat{L}_{d,\lambda}^{+} \otimes \widehat{L}_{d,\lambda}^{-}
\to \mathbbm{k}.
\]
By the first part, this is equivalent to a perfect pairing 
\[
M_d^{\ev_a}\otimes M_d^{\ev'_{a^{-1}}}\to \mathbbm{k},
\]
for $a=\lambda (-q)^{d-2}$. 
\end{proof}

\begin{rem} We claim no originality w.r.t. Theorem~\ref{thm:isosandpairing}, but we do not know of any reference in the literature where one can find it explicitly, which is why 
we have proved it here.  
\end{rem}


\section{Reminders on Soergel categories}\label{sec:soergel}

In this section we briefly recall the definition of the diagrammatic Soergel category 
of non-extended and extended affine type $A$ and 
finite type $A$, but before we do that we start with a brief section 
on graded categories and categories with shift. 

\subsection{Graded categories and categories with shift}\label{sec:graded-and-shifted-stuff}
All categories in this paper are assumed to be essentially small, meaning that they are equivalent to small categories, so set-theoretic questions play no role.  

We call a $\mathbb{C}$-linear category $\mathcal{A}$ {\em graded} if it is enriched over the category of $\mathbb{Z}$-graded vector spaces, and we call a $\mathbb{C}$-linear functor between such graded categories {\em degree-preserving} if it preserves the degrees of homogeneous morphisms.   

We say that a $\mathbb{C}$-linear category $\mathcal{A}$ has a {\em shift} (or, alternatively, 
that it is a {\em category with shift}) if there is a 
$\mathbb{C}$-linear automorphism $\langle 1\rangle$ of $\mathcal{A}$. If such a shift exists, 
we define $\langle r\rangle$ as the composite of $r$ copies of $\langle 1\rangle$ for any $r\in \mathbb{Z}_{\geq 0}$, 
and $-r$ copies of the inverse of $\langle 1\rangle$ for any $r\in \mathbb{Z}_{\leq 0}$. By definition, therefore,  
we have $\langle r+s\rangle =\langle r\rangle\circ\langle s \rangle$, for all $r,s\in \mathbb{Z}$, and $\langle 0\rangle=\mathrm{Id}_{\mathcal{A}}$.   

Given a graded category $\mathcal{A}$, let $\mathcal{A}^{\mathrm{sh}}$ be the 
associated $\mathbb{C}$-linear category with shift, whose objects are formal integer 
shifts of objects in $\mathcal{A}$ and whose hom-spaces are defined by 
\[
\mathcal{A}^{\mathrm{sh}}\left(X\langle r\rangle, Y\langle s\rangle\right):=\mathcal{A}\left(X,Y\right)_{s-r}
\] 
for every $X,Y\in \mathcal{A}$ and $r,s\in \mathbb{Z}$. 
Note that $\mathcal{A}^{\mathrm{sh}}$ is no longer a graded category.  
If the Hom-spaces of $\mathcal{A}$ are finite-dimensional in every degree, then the hom-spaces of 
$\mathcal{A}^{\mathrm{sh}}$ are finite-dimensional.  

Given two graded categories $\mathcal{A}$ and $\mathcal{B}$, any 
degree-preserving, $\mathbb{C}$-linear functor $F\colon \mathcal{A}\to \mathcal{B}$ induces a unique $\mathbb{C}$-linear functor $F\colon \mathcal{A}^{\mathrm{sh}}\to \mathcal{B}^{\mathrm{sh}}$, denoted by the same symbol, which commutes with the shifts.  

Conversely, given any $\mathbb{C}$-linear category $\mathcal{A}$ with shift, 
let $\mathcal{A}^{\mathrm{gr}}$ be the associated graded category with shift,  
whose objects are those of $\mathcal{A}$ and whose graded Hom-spaces are defined by
\[
\mathcal{A}^{\mathrm{gr}}\left(X,Y\right):=\bigoplus_{s\in \mathbb{Z}}\mathcal{A}\left(X,Y\langle s\rangle\right),
\]
for any $X,Y\in \mathcal{A}$. 

Given two $\mathbb{C}$-linear categories $\mathcal{A}$ and $\mathcal{B}$ with shifts, 
any $\mathbb{C}$-linear functor $F\colon \mathcal{A}\to \mathcal{B}$ commuting with the shifts induces a unique degree-preserving, $\mathbb{C}$-linear functor $F\colon \mathcal{A}^{\mathrm{gr}}\to \mathcal{B}^{\mathrm{gr}}$, denoted by the same symbol.  

Thus $(-)^{\mathrm{sh}}$ and $(-)^{\mathrm{gr}}$ define a pair of $2$-functors between the 
$2$-category of graded categories and the $2$-category of $\mathbb{C}$-linear 
categories with shift. It is not hard to show, see e.g.~\cite[Proposition 11.9]{e-m-t-w}, 
that $(-)^{\mathrm{sh}}$ is left adjoint to $(-)^{\mathrm{gr}}$, i.e., that there is a functorial isomorphism 
\[
\mathrm{Fun}\left(\mathcal{A}^{\mathrm{sh}}, \mathcal{B}\right)\cong \mathrm{Fun}\left(\mathcal{A}, 
\mathcal{B}^{\mathrm{gr}}\right)
\] 
for $\mathcal{A}$ a graded category and $\mathcal{B}$ a $\mathbb{C}$-linear category with shift. Here the first functor category is between categories with shift and the second between graded 
categories. 

For more details on graded categories and categories with shift, and also on additive closures and idempotent completions (a.k.a. Karoubi closures/envelopes), see e.g.~\cite[Sections 11.2.1-11.2.4]{e-m-t-w}. 


\subsection{Soergel calculus in finite and non-extended affine type A}\label{sec:soergeldiagrammatics-one}

The finite type $A$ diagrammatic Soergel calculus was introduced by Elias--Khovanov~\cite{elias-khovanov} 
and generalized to all Coxeter types by Elias--Williamson~\cite{elias-williamson-2}. 
The extended affine Soergel calculus was first defined in~\cite{mackaay-thiel} and studied more systematically 
in~\cite{elias2018}. We refer to the latter two papers for more details. 
For the specialists, we remark that we use the so-called {\em root span realization} of the Cartan datum of finite and affine type A below.  

Denote by $S=\{s_i\mid i\in \widehat{I}\}$ 
the set of simple reflections of $\widehat{\Sy}_d$. 
The \emph{diagrammatic Bott-Samelson category} of type $\widehat{A}_{d-1}$, denoted 
$\widehat{\eBS}_d$, is the $\mathbb{Z}$-graded, $\mathbb{C}$-linear, additive, 
monoidal category whose objects are formal finite direct sums of finite words in the alphabet 
$S$, and whose graded vector spaces of morphisms are defined below 
in terms of homogeneous generating diagrams and relations. In general, we can write 
the objects as vectors of words and morphisms as matrices of equivalence classes of diagrams.   

As usual, we will color the strands to facilitate the reading of the diagrams. 
These colors correspond to the elements of 
$\widehat{I}$, so henceforth we will also refer to those elements as colors. When there are too many different colors in a diagram, the colors are sometimes 
indicated by labels next to the strands. We 
say that two colors $i,j\in \widehat{I}$ are {\em adjacent} if $i \equiv j\pm 1\bmod d$ 
and that they are {\em distant} otherwise.  The generating diagrams are
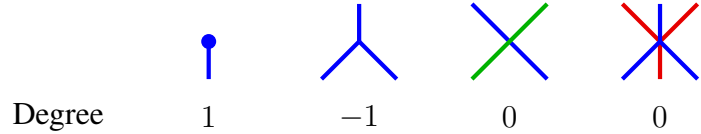
\begin{equation*}
\xy (0,0)*{
\tikzdiagc[scale=1]{
\begin{scope}[yscale=-.5,xscale=.5,shift={(5,-2)}] 
  \draw[ultra thick,blue] (-1,0) -- (-1, 1)node[pos=0, tikzdot]{};
\end{scope}
\begin{scope}[yscale=.5,xscale=.5,shift={(8,2)}] 
   \draw[ultra thick,blue] (0,0)-- (0, 1); \draw[ultra thick,blue] (-1,-1) -- (0,0); \draw[ultra thick,blue] (1,-1) -- (0,0);
\end{scope}
\begin{scope}[scale=1,shift={(6,1)}] 
\draw[ultra thick,blue] (.5,-.5)  -- (-.5,.5);
\draw[ultra thick,mygreen] (-.5,-.5) -- (.5,.5);
\end{scope}
\begin{scope}[scale=.5,shift={(16,2)}] 
  \draw[ultra thick,myred] (0,-1) -- (0,0);\draw[ultra thick,myred] (0,0) -- (-1, 1);\draw[ultra thick,myred] (0,0) -- (1, 1);
  \draw[ultra thick,blue] (0,0)-- (0, 1); \draw[ultra thick,blue] (-1,-1) -- (0,0); \draw[ultra thick,blue] (1,-1) -- (0,0);
\end{scope}
\node at (0,0) {Degree};
\node at (2,0) {$1$};
\node at (4,0) {$-1$};
\node at (6,0) {$0$};
\node at (8,0) {$0$};
}}\endxy
\end{equation*}
and the diagrams obtained from these by a rotation of $180$ degrees (which have the 
same degrees). The colors of the 
4-valent vertices are assumed to be distant, whereas those of the 6-valent vertices are assumed 
to be adjacent. 

Diagrams can be stacked vertically (composition of morphisms) and juxtaposed horizontally (monoidal product of morphisms), while adding the degrees, and are subject to the relations 
below. We denote by $\id_{X}$ the identity morphism of $X$ and write $fg$ for the 
monoidal product of morphisms $f$ and $g$ (or, equivalently, horizontal composition when considering the monoidal category as a one-object bicategory). We also assume isotopy invariance and cyclicity, meaning that closed parts of the 
diagrams can be moved around freely in the plane as long as they do not cross any other strands 
and the boundary is fixed, and all diagrams can be bent and rotated and the bent and rotated 
versions of the relations also hold. 

\begin{itemize}
\item Relations involving one color:

\begingroup\allowdisplaybreaks
\begin{gather}\label{eq:relhatSfirst}
\xy (0,0)*{
\tikzdiagc[scale=.4,yscale=.7]{
  \draw[ultra thick,blue] (0,-1.75) -- (0,1.75);
  \draw[ultra thick,blue] (0,0) -- (1,0)node[pos=1, tikzdot]{};  
}}\endxy
=\ 
\xy (0,0)*{
\tikzdiagc[scale=.4,yscale=.7]{
  \draw[ultra thick,blue] (0,-1.75) -- (0,1.75);
}}\endxy
\\[1ex] \label{eq:relhatSsecond}
  \xy (0,.05)*{
\tikzdiagc[scale=.4,yscale=.7]{
  \draw[ultra thick,blue] (0,-1) -- (0,1);
  \draw[ultra thick,blue] (0,1) -- (-1,2);  \draw[ultra thick,blue] (0,1) -- (1,2);
  \draw[ultra thick,blue] (0,-1) -- (-1,-2);  \draw[ultra thick,blue] (0,-1) -- (1,-2);
}}\endxy
=
  \xy (0,.05)*{
\tikzdiagc[yscale=.4,xscale=.3]{
  \draw[ultra thick,blue] (-1,0) -- (1,0);
  \draw[ultra thick,blue] (-2,-1) -- (-1,0);  \draw[ultra thick,blue] (-2,1) -- (-1,0);
  \draw[ultra thick,blue] ( 2,-1) -- ( 1,0);  \draw[ultra thick,blue] ( 2,1) -- ( 1,0);
}}\endxy
\\[1ex] \label{eq:lollipop}
  \xy (0,0)*{
\tikzdiagc[scale=0.9]{
\draw[ultra thick,blue] (0,0) circle  (.3);
\draw[ultra thick,blue] (0,-.8) --  (0,-.3);
}}\endxy
  \ = 0
\\[1ex] \label{eq:relHatSlast}
\xy (0,0)*{
  \tikzdiagc[yscale=0.5,xscale=.5]{
    \draw[ultra thick,blue] (0,-.35) -- (0,.35)node[pos=0, tikzdot]{} node[pos=1, tikzdot]{};
  \draw[ultra thick,blue] (.6,-1)-- (.6, 1); 
}}\endxy
\ +\ 
\xy (0,0)*{
  \tikzdiagc[yscale=0.5,xscale=-.5]{
    \draw[ultra thick,blue] (0,-.35) -- (0,.35)node[pos=0, tikzdot]{} node[pos=1, tikzdot]{};
  \draw[ultra thick,blue] (.6,-1)-- (.6, 1); 
}}\endxy\
=
2\,\  
\xy (0,0)*{
  \tikzdiagc[yscale=0.5,xscale=-.5]{
\draw[ultra thick,blue] (0,-1) -- (0,-.4)node[pos=1, tikzdot]{};
\draw[ultra thick,blue] (0,.4) -- (0,1)node[pos=0, tikzdot]{};
}}\endxy
\end{gather}
\endgroup

\item Relations involving two distant colors:
\begingroup\allowdisplaybreaks  
\begin{gather}
\label{eq:Scat-Rtwo}
  \xy (0,0)*{
  \tikzdiagc[yscale=1.7,xscale=1.1]{
    \draw[ultra thick,blue] (0,0) ..controls (0,.25) and (.65,.25) .. (.65,.5) ..controls (.65,.75) and (0,.75) .. (0,1);
\begin{scope}[shift={(.65,0)}]
    \draw[ultra thick,mygreen] (0,0) ..controls (0,.25) and (-.65,.25) .. (-.65,.5) ..controls (-.65,.75) and (0,.75) .. (0,1);
\end{scope}
}}\endxy
= \ 
\xy (0,0)*{
  \tikzdiagc[yscale=1.7,xscale=-1.1]{
\draw[ultra thick, blue] (.65,0) -- (.65,1);
\draw[ultra thick,mygreen] (0,0) -- (0,1);
  }}\endxy
\\[1ex]\label{eq:Scat-dotslide}
\xy (0,0)*{
\tikzdiagc[scale=1]{
\draw[ultra thick,blue] (.3,-.3) -- (-.5,.5)node[pos=0, tikzdot]{};
\draw[ultra thick,mygreen] (-.5,-.5) -- (.5,.5);
}}\endxy
=
\xy (0,0)*{
\tikzdiagc[scale=1]{
\draw[ultra thick,blue] (-.5,.5) -- (-.2,.2)node[pos=1, tikzdot]{};
\draw[ultra thick,mygreen] (-.5,-.5) -- (.5,.5);
}}\endxy
\\[1ex]\label{eq:Scat-trivslide}
\xy (0,0)*{
  \tikzdiagc[yscale=0.5,xscale=.5]{
  \draw[ultra thick,blue] (0,0)-- (0, 1); \draw[ultra thick,blue] (-1,-1) -- (0,0); \draw[ultra thick,blue] (1,-1) -- (0,0);
\draw[ultra thick,mygreen] (-1,0) ..controls (-.25,.75) and (.25,.75) .. (1,0);
}}\endxy
=
\xy (0,0)*{
  \tikzdiagc[yscale=0.5,xscale=.5]{
  \draw[ultra thick,blue] (0,0)-- (0, 1); \draw[ultra thick,blue] (-1,-1) -- (0,0); \draw[ultra thick,blue] (1,-1) -- (0,0);
\draw[ultra thick,mygreen] (-1,0) ..controls (-.25,-.75) and (.25,-.75) .. (1,0);
}}\endxy
\end{gather}
\endgroup

\item Relations involving two adjacent colors:

\begingroup\allowdisplaybreaks  
\begin{gather}\label{eq:6vertexdot}
\xy (0,0)*{
  \tikzdiagc[yscale=0.5,xscale=.5]{
  \draw[ultra thick,myred] (0,-.75) -- (0,0)node[pos=0, tikzdot]{};\draw[ultra thick,myred] (0,0) -- (-1, 1);\draw[ultra thick,myred] (0,0) -- (1, 1);
  \draw[ultra thick,blue] (0,0)-- (0, 1); \draw[ultra thick,blue] (-1,-1) -- (0,0); \draw[ultra thick,blue] (1,-1) -- (0,0);
}}\endxy
\ =\ \ 
\xy (0,0)*{
  \tikzdiagc[yscale=0.5,xscale=.5]{
  \draw[ultra thick,myred] (-1,0) -- (-1, 1)node[pos=0, tikzdot]{};\draw[ultra thick,myred] (1,0) -- (1, 1)node[pos=0, tikzdot]{};
  \draw[ultra thick,blue] (0,0)-- (0, 1); \draw[ultra thick,blue] (-1,-1) -- (0,0); \draw[ultra thick,blue] (1,-1) -- (0,0);
}}\endxy
\ \ + \
\xy (0,0)*{
\tikzdiagc[yscale=0.6,xscale=.8]{
 \draw[ultra thick,blue] (1.5,0.1) -- (1.5, 0.4)node[pos=0, tikzdot]{};
 \draw[ultra thick,myred] (.9,.4) .. controls (1.2,-.45) and (1.8,-.45) .. (2.1,.4);
 \draw[ultra thick,blue] ( .9,-1.2) .. controls (1.2,-.35) and (1.8,-.35) .. (2.1,-1.2);
}}\endxy
\\[1ex]\label{eq:braidmoveB}
\xy (0,.05)*{
\tikzdiagc[scale=0.4,yscale=1]{
  \draw[ultra thick,myred] (-1,-2) -- (-1,2);
  \draw[ultra thick,myred] ( 1,-2) -- ( 1,2);
  \draw[ultra thick,blue] (0,-2)  --  (0,2); 
}}\endxy
=
\xy (0,.05)*{
\tikzdiagc[scale=0.4,yscale=1]{
  \draw[ultra thick,myred] (0,-1) -- (0,1);
  \draw[ultra thick,myred] (0,1) -- (-1,2);  \draw[ultra thick,myred] (0,1) -- (1,2);
  \draw[ultra thick,myred] (0,-1) -- (-1,-2);  \draw[ultra thick,myred] (0,-1) -- (1,-2);
  \draw[ultra thick,blue] (0,1)  --  (0,2);  \draw[ultra thick,blue] (0,-1)  --  (0,-2);
  \draw[ultra thick,blue] (0,1)  ..controls (-.95,.25) and (-.95,-.25) ..  (0,-1);
  \draw[ultra thick,blue] (0,1)  ..controls ( .95,.25) and ( .95,-.25) ..  (0,-1);
}}\endxy
-
\xy (0,.05)*{
\tikzdiagc[scale=0.4,yscale=1]{
  \draw[ultra thick,myred] (0,-.6) -- (0,.6);
  \draw[ultra thick,myred] (0,.6) .. controls (-.25,.6) and (-1,1).. (-1,2);
  \draw[ultra thick,myred] (0,.6) .. controls (.25,.6) and (1,1) .. (1,2);
  \draw[ultra thick,myred] (0,-.6) .. controls (-.25,-.6) and (-1,-1).. (-1,-2);
  \draw[ultra thick,myred] (0,-.6) .. controls (.25,-.6) and (1,-1) .. (1,-2);
  \draw[ultra thick,blue] (0,1.25)  --  (0,2)node[pos=0, tikzdot]{};  \draw[ultra thick,blue] (0,-1.25)  --  (0,-2)node[pos=0, tikzdot]{};
 }}\endxy
\\[1ex]
\label{eq:stroman}
\xy (0,.05)*{
\tikzdiagc[scale=0.4,yscale=1]{
  \draw[ultra thick,myred] (0,-1) -- (0,1);
  \draw[ultra thick,myred] (0,1) -- (-1,2);  \draw[ultra thick,myred] (0,1) -- (1,2);
  \draw[ultra thick,myred] (0,-1) -- (-1,-2);  \draw[ultra thick,myred] (0,-1) -- (1,-2);
  \draw[ultra thick,blue] (0,1)  --  (0,2);  \draw[ultra thick,blue] (0,-1)  --  (0,-2);
  \draw[ultra thick,blue] (0,1)  ..controls (-.95,.25) and (-.95,-.25) ..  (0,-1);
  \draw[ultra thick,blue] (0,1)  ..controls ( .95,.25) and ( .95,-.25) ..  (0,-1);
  \draw[ultra thick,blue] (-2,0)  --  (-.75,0);  \draw[ultra thick,blue] (2,0)  --  (.75,0);
}}\endxy
=
  \xy (0,.05)*{
\tikzdiagc[scale=0.4,yscale=1]{
  \draw[ultra thick,myred] (-1,0) -- (1,0);
  \draw[ultra thick,myred] (-2,-1) -- (-1,0);  \draw[ultra thick,myred] (-2,1) -- (-1,0);
  \draw[ultra thick,myred] ( 2,-1) -- ( 1,0);  \draw[ultra thick,myred] ( 2,1) -- ( 1,0);
  \draw[ultra thick,blue] (-1,0)  ..controls (-.5,-.8) and (.5,-.8) ..  (1,0);
  \draw[ultra thick,blue] (-1,0)  ..controls (-.5, .8) and (.5, .8) ..  (1,0);
  \draw[ultra thick,blue] (0,.65)  --  (0,1.8);  \draw[ultra thick,blue] (0,-.65)  --  (0,-1.8);
  \draw[ultra thick,blue] (-2,0)  --  (-1,0);  \draw[ultra thick,blue] (2,0)  --  (1,0);
}}\endxy
\\[1ex]\label{eq:forcingdumbel-i-iminus}
\xy (0,0)*{
  \tikzdiagc[yscale=0.5,xscale=.5]{
    \draw[ultra thick,myred] (0,-.35) -- (0,.35)node[pos=0, tikzdot]{} node[pos=1, tikzdot]{};
  \draw[ultra thick,blue] (.6,-1)-- (.6, 1); 
}}\endxy
-
\xy (0,0)*{
  \tikzdiagc[yscale=0.5,xscale=-.5]{
    \draw[ultra thick,myred] (0,-.35) -- (0,.35)node[pos=0, tikzdot]{} node[pos=1, tikzdot]{};
  \draw[ultra thick,blue] (.6,-1)-- (.6, 1); 
}}\endxy
\ =\frac{1}{2}
\biggl(\ 
\xy (0,0)*{
  \tikzdiagc[yscale=0.5,xscale=-.5]{
    \draw[ultra thick,blue] (0,-.35) -- (0,.35)node[pos=0, tikzdot]{} node[pos=1, tikzdot]{};
  \draw[ultra thick,blue] (.6,-1)-- (.6, 1); 
}}\endxy
-
\xy (0,0)*{
  \tikzdiagc[yscale=0.5,xscale=.5]{
    \draw[ultra thick,blue] (0,-.35) -- (0,.35)node[pos=0, tikzdot]{} node[pos=1, tikzdot]{};
  \draw[ultra thick,blue] (.6,-1)-- (.6, 1); 
}}\endxy\ 
\biggr)
\end{gather}
\endgroup


\item Relation involving three distant colors:
\begin{equation}
\xy (0,0)*{
  \tikzdiagc[yscale=0.5,xscale=.5]{
 \draw[ultra thick,blue] (-1,-1) -- (1,1);
 \draw[ultra thick,mygreen] (1,-1) -- (-1,1);
 \draw[ultra thick,orange] (-1,0) ..controls (-.25,.75) and (.25,.75) .. (1,0);
}}\endxy
=
\xy (0,0)*{
  \tikzdiagc[yscale=0.5,xscale=.5]{
 \draw[ultra thick,blue] (-1,-1) -- (1,1);
 \draw[ultra thick,mygreen] (1,-1) -- (-1,1);
\draw[ultra thick,orange] (-1,0) ..controls (-.25,-.75) and (.25,-.75) .. (1,0);
}}\endxy
\end{equation}

\item Relation involving distant dumbbells:
\begin{equation}\label{eq:distantdumbbells}
\xy (0,0)*{
  \tikzdiagc[yscale=0.5,xscale=.5]{
    \draw[ultra thick,mygreen] (0,-.35) -- (0,.35)node[pos=0, tikzdot]{} node[pos=1, tikzdot]{};
  \draw[ultra thick,blue] (.6,-1)-- (.6, 1); 
}}\endxy
-
\xy (0,0)*{
  \tikzdiagc[yscale=0.5,xscale=-.5]{
    \draw[ultra thick,mygreen] (0,-.35) -- (0,.35)node[pos=0, tikzdot]{} node[pos=1, tikzdot]{};
  \draw[ultra thick,blue] (.6,-1)-- (.6, 1); 
}}\endxy
\ = 0
\end{equation}

\item Relation involving two adjacent colors and one distant from the other two:
\begin{equation}\label{eq:sixv-dist}
\xy (0,0)*{
  \tikzdiagc[yscale=0.5,xscale=.5]{
  \draw[ultra thick,myred] (0,-1) -- (0,0);\draw[ultra thick,myred] (0,0) -- (-1, 1);\draw[ultra thick,myred] (0,0) -- (1, 1);
  \draw[ultra thick,blue] (0,0)-- (0, 1); \draw[ultra thick,blue] (-1,-1) -- (0,0); \draw[ultra thick,blue] (1,-1) -- (0,0);
\draw[ultra thick,mygreen] (-1,0) ..controls (-.25,.75) and (.25,.75) .. (1,0);  
}}\endxy
\ =\ \ 
\xy (0,0)*{
  \tikzdiagc[yscale=0.5,xscale=.5]{
  \draw[ultra thick,myred] (0,-1) -- (0,0);\draw[ultra thick,myred] (0,0) -- (-1, 1);\draw[ultra thick,myred] (0,0) -- (1, 1);
  \draw[ultra thick,blue] (0,0)-- (0, 1); \draw[ultra thick,blue] (-1,-1) -- (0,0); \draw[ultra thick,blue] (1,-1) -- (0,0);
\draw[ultra thick,mygreen] (-1,0) ..controls (-.25,-.75) and (.25,-.75) .. (1,0);  
}}\endxy
\end{equation}

\item Relation involving three colors such that one of them is adjacent to the other two:
\begin{equation}\label{eq:relhatSlast}
\xy (0,.05)*{
\tikzdiagc[scale=0.5,yscale=1]{
  \draw[ultra thick,myred] (0,-1) -- (0,1);
  \draw[ultra thick,myred] (0,1) -- (-1,2);
  \draw[ultra thick,myred] (0,1) -- (2,2);
  \draw[ultra thick,myred] (0,-1) -- (-2,-2);  \draw[ultra thick,myred] (0,-1) -- (1,-2);
  \draw[ultra thick,blue] (0,1)  --  (0,2);  \draw[ultra thick,blue] (0,-1)  --  (0,-2);
  \draw[ultra thick,blue] (0,1) -- (1,0); \draw[ultra thick,blue] (0,1) -- (-1,0);
  \draw[ultra thick,blue] (0,-1) -- (1,0); \draw[ultra thick,blue] (0,-1) -- (-1,0);
  \draw[ultra thick,blue] (-2,0) -- (-1,0);\draw[ultra thick,blue] (2,0) -- (1,0);
\draw[ultra thick,mygreen] (-1,0) -- (1,0);
\draw[ultra thick,mygreen] (-1,0) -- (-1,-2);\draw[ultra thick,mygreen] (-1,0) -- (-2,2);   
\draw[ultra thick,mygreen] (1,0) -- (2,-2);\draw[ultra thick,mygreen] (1,0) -- (1,2);   
}}\endxy
=
\xy (0,.05)*{
\tikzdiagc[scale=0.5,yscale=1]{
  \draw[ultra thick,mygreen] (0,-1) -- (0,1);
  \draw[ultra thick,mygreen] (0,1) -- (1,2);
  \draw[ultra thick,mygreen] (0,1) -- (-2,2);
  \draw[ultra thick,mygreen] (0,-1) -- (2,-2);
  \draw[ultra thick,mygreen] (0,-1) -- (-1,-2);
  \draw[ultra thick,blue] (0,1)  --  (0,2);
  \draw[ultra thick,blue] (0,-1)  --  (0,-2);
  \draw[ultra thick,blue] (0,1) -- (1,0);
  \draw[ultra thick,blue] (0,1) -- (-1,0);
  \draw[ultra thick,blue] (0,-1) -- (1,0);
  \draw[ultra thick,blue] (0,-1) -- (-1,0);
  \draw[ultra thick,blue] (-2,0) -- (-1,0);
  \draw[ultra thick,blue] (2,0) -- (1,0);
\draw[ultra thick,myred] (1,0) -- (-1,0);
\draw[ultra thick,myred] (1,0) -- (1,-2);
\draw[ultra thick,myred] (1,0) -- (2,2);   
\draw[ultra thick,myred] (-1,0) -- (-2,-2);
\draw[ultra thick,myred] (-1,0) -- (-1,2);   
}}\endxy
\end{equation}

\end{itemize}  

Note that the empty word is the identity object in $\widehat{\eBS}_d$ and its endomorphisms are 
the closed diagrams, which by the relations above are equal to polynomials 
in the colored dumbbells 
\begin{equation*}
\xy (0,0)*{
\tikzdiagc[scale=1]{
\begin{scope}[yscale=-.5,xscale=.5] 
  \draw[ultra thick,blue] (-1,-.4) -- (-1, .4)node[pos=0, tikzdot]{} node[pos=1, tikzdot]{};
\end{scope}
}}\endxy
\end{equation*}
As each dumbbell has degree $2$, the degree of any polynomial in these dumbbells,  
as a morphism in $\widehat{\eS_{BS}}$, is twice its polynomial degree. From now on, 
we denote this polynomial algebra by $R$.  

Note further that, by relations \eqref{eq:relHatSlast}, \eqref{eq:forcingdumbel-i-iminus} 
and \eqref{eq:distantdumbbells}, the morphism 
\begin{equation}\label{eq:dumbbellsum}
\sum_{i=0}^{d-1} \,\xy (0,0)*{
\tikzdiagc[scale=1]{
\begin{scope}[yscale=-.5,xscale=.5] 
  \draw[ultra thick,blue] (-1,-.4) -- (-1, .4)node[pos=0, tikzdot]{} node[pos=1, tikzdot]{} 
node at (-0.5,0) {\small $i$} ;
\end{scope}
}}\endxy
\end{equation}
is central, in the sense that it can be slid through all diagrams 
(i.e. it commutes horizontally with all morphisms). Note that this morphism is equal to 
$\fbox{y}$ (up to sign, depending on conventions) in~\cite{mackaay-thiel}, because it is  
equal to the sum of all simple roots.

Let $\widehat{\eBS}_d^{\mathrm{sh}}$ be the category with shift associated to 
$\widehat{\eBS}_d$, see Section~\ref{sec:graded-and-shifted-stuff}.

\begin{defn}
The {\em diagrammatic Soergel category} $\Saff_d$ is the 
idempotent completion of the diagrammatic Bott-Samelson category with shift 
$\widehat{\eBS}_d^{\mathrm{sh}}$. 
\end{defn}

\begin{rem}\label{rem:shifts-etc1} In the following sections, we sometimes state and prove diagrammatic equations in $\widehat{\eBS}_d$, in which case there 
are no shifts for the source and target objects, instead of $\Saff_d$, in which case 
the source and target objects are carefully shifted. This is just to simplify notation and makes no essential difference in our case. As long as the equations in $\widehat{\eBS}_d$ are between homogeneous diagrams of the same degree, they give rise to an equality between morphisms in $\Saff_d$, which is the key point. 
\end{rem}

The diagrammatic Bott-Samelson category $\widehat{\eBS}_d$ is equivalent to the algebraic 
category of Bott-Samelson bimodules and bimodule maps 
and the diagrammatic Soergel category $\Saff_d$ is equivalent to the algebraic category 
of Soergel bimodules and degree-preserving bimodule maps, 
see~\cite[Theorem 6.28]{elias-williamson-2}. 
For convenience, we will therefore denote the objects of $\widehat{\eBS}_d$ by 
$\rB_{\underline{w}}=\rB_{s_{i_1}}\cdots \rB_{s_{i_{\ell}}}$, where $\underline{w}=s_{i_1}\cdots s_{i_{\ell}}$ is a finite word in the alphabet $S$. In particular, the monoidal product is given by $\rB_{\underline{u}}\rB_{\underline{v}} =
\rB_{\underline{uv}}$, where $\underline{uv}$ is the 
concatenation of the words $\underline{u}$ and $\underline{v}$. 

Let us also recall the so-called {\em Categorification Theorem}, due to Soergel in finite type $A$, 
to H\"arterich~\cite{harterich} in affine type $A$ and to Elias--Williamson~\cite{elias-williamson-1, elias-williamson-2} in general Coxeter type. 

\begin{thm}\label{thm:categorification}
For any $w\in \Saff_d$ and rex $\underline{w}=s_{i_1}\cdots s_{i_\ell}$ of $w$, 
there is an indecomposable object $\rB_w\in \Saff_d$, independent of the choice of rex, such that
\[
\rB_{\underline{w}}\cong \rB_w \oplus \bigoplus_{u \prec w} \rB_u^{\oplus h_{w,u}},
\] 
where $\prec$ is the Bruhat order in $\Saff_d$ and $h_{w,u}\in \mathbb{N}[q,q^{-1}]$ 
is the graded multiplicity of $\rB_u$ in the decomposition of $\rB_{\underline{w}}$. 

Moreover, the $\mathbb{Z}[q,q^{-1}]$-linear map 
\begin{gather*}
\affh^{\mathbb{Z}[q,q^{-1}]} \to [\widehat{\eS}_d]_{\oplus}\\
b_w\mapsto \rB_w, \quad w \in  \Saff_d
\end{gather*}
is an isomorphism of algebras, where 
$\affh^{\mathbb{Z}[q,q^{-1}]}$ is the integral form of $\affh$.
\end{thm}

Let $\Saff_d^{\mathrm{gr}}$ be the graded monoidal category associated to $\Saff_d$, see 
Section~\ref{sec:graded-and-shifted-stuff}. For every $u,v\in \Saff$, the graded Hom-space 
\[
\Saff_d^{\mathrm{gr}}\left(\rB_u,\rB_v\right)=\bigoplus_{t\in \mathbb{Z}}\Saff_d\left(\rB_u,\rB_v
\langle t\rangle \right)
\] 
is a free left (or right) graded $R$-module of finite graded rank, given 
by {\em Soergel's Hom-formula}:  
\begin{equation}\label{eq:Soergelhom}
\mathrm{grk}_{R}\left(\Saff_d^{\mathrm{gr}}\left(\rB_u,\rB_v\right)\right)=(b_u,b_v),
\end{equation}
where $(-,-)$ is the well-known sesquilinear form on $\affh$, see 
e.g.~\cite[Section 2.4 and Theorem 3.15]{elias-williamson-2}.

\begin{defn}
The diagrammatic Bott-Samelson category and the diagrammatic Soergel category of finite type $A_{d-1}$, denoted $\eBS_d$ and $\eS_d$ respectively, are defined as $\widehat{\eBS}_d$ and 
$\Saff_d$ but only using the colors $I$. 
\end{defn}

Note that $\eBS_d$ and $\eS_d$ are monoidal subcategories of $\widehat{\eBS}_d$ and 
$\Saff_d$, respectively, but that the natural embeddings are not full because 
e.g. the $0$-colored dumbbell is not a morphism in $\eBS_d$ and $\eS_d$.    


\subsection{Soergel calculus in extended affine type $A$}\label{sec:soergeldiagrammatics-two}

In this subsection we briefly sketch how to enhance $\widehat{\eBS}_d$ and $\Saff_d$ 
to get the extended diagrammatic Soergel category of type $\widehat{A}_{d-1}$, denoted 
$\BSext_d$ and $\Sext_d$, which were introduced in~\cite{mackaay-thiel} and further studied 
in~\cite{elias2018}. We refer to those two papers for more details. 

The objects of $\BSext_d$ are formal direct sums of words in the 
alphabet $S\cup\{\rho,\rho^{-1}\}$. Because of 
the link with algebraic bimodules, we write $\rB_{\rho}^n$ for $\rho^n$, for any 
$n\in \mathbb{Z}$.

There are also new generating diagrams, all of degree zero, involving oriented strands. The generators involving only oriented strands are 
\begin{equation}\label{eq:orientedcoloredgens1}
\xy (0,0)*{
  \tikzdiagc[yscale=0.9,baseline={([yshift=-.8ex]current bounding box.center)}]{
  \draw[ultra thick,black,-to] (1.5,-.5) -- (1.5, .5);
  }
}\endxy
\mspace{60mu}
\xy (0,0)*{
  \tikzdiagc[yscale=0.9,baseline={([yshift=-.8ex]current bounding box.center)}]{
  \draw[ultra thick,black,to-] (1.5,-.5) -- (1.5, .5);
  }
}\endxy
\mspace{60mu}
\xy (0,.55)*{
\tikzdiagc[yscale=0.9]{
\draw[ultra thick,black,to-] (1,.4) .. controls (1.2,-.4) and (1.8,-.4) .. (2,.4);
}}\endxy
\mspace{60mu}
  \xy (0,.55)*{
\tikzdiagc[yscale=0.9]{
\draw[ultra thick,black,-to] (1,.4) .. controls (1.2,-.4) and (1.8,-.4) .. (2,.4);
    }}\endxy
\mspace{60mu}
  \xy (0,0)*{
\tikzdiagc[yscale=-0.9]{
\draw[ultra thick,black,to-] (1,.4) .. controls (1.2,-.4) and (1.8,-.4) .. (2,.4);
            }}\endxy
\mspace{60mu}
  \xy (0,0)*{
\tikzdiagc[yscale=-0.9]{
\draw[ultra thick,black,-to] (1,.4) .. controls (1.2,-.4) and (1.8,-.4) .. (2,.4);
            }}\endxy
\end{equation}  
and the generating diagrams involving oriented strands and adjacent colored strands are 
\begin{equation}\label{eq:orientedcoloredgens2}
  \xy (0,-2.25)*{
\tikzdiagc[scale=1]{
\draw[ultra thick,myred] (.5,-.5)node[below] {\tiny $i-1$} -- (0,0);
\draw[ultra thick,blue] (-.5,.5) node[above] {\tiny $i$} -- (0,0);
\draw[ultra thick,black,-to] (-.5,-.5) -- (.5,.5);
  }}\endxy
\mspace{60mu}  
  \xy (0,-2.25)*{
\tikzdiagc[scale=1,xscale=-1]{
\draw[ultra thick,blue] (.5,-.5)node[below] {\tiny $i$} -- (0,0);
\draw[ultra thick,myred] (-.5,.5) node[above] {\tiny $i-1$}-- (0,0);
\draw[ultra thick,black,-to] (-.5,-.5) -- (.5,.5);
  }}\endxy
\mspace{60mu}
\xy (0,-2.25)*{
\tikzdiagc[scale=1]{
\draw[ultra thick,blue] (.5,-.5)node[below] {\tiny $i$} -- (0,0);
\draw[ultra thick,myred] (-.5,.5) node[above] {\tiny $i-1$} -- (0,0);
\draw[ultra thick,black,-to] (.5,.5) -- (-.5,-.5);
  }}\endxy
\mspace{60mu}
  \xy (0,-2.25)*{
\tikzdiagc[scale=1,yscale=-1]{
\draw[ultra thick,blue] (.5,-.5) node[above] {\tiny $i$} -- (0,0);
\draw[ultra thick,myred] (-.5,.5)node[below] {\tiny $i-1$} -- (0,0);
\draw[ultra thick,black,-to] (-.5,-.5) -- (.5,.5);
  }}\endxy
\end{equation}

The new morphisms satisfy the following relations, where we again assume isotopy invariance and cyclicity.
\begin{itemize}
\item Relations involving only oriented strands:
\begingroup\allowdisplaybreaks
  \begin{gather}\label{eq:orloop}
  \xy (0,0)*{
\tikzdiagc[yscale=-0.9]{
\draw[ultra thick,black] (0,0) circle  (.65);\draw [ultra thick,black,-to] (.65,0) --(.65,0);
  }}\endxy
  \ = 1 =\ 
    \xy (0,0)*{
\tikzdiagc[yscale=-0.9]{
\draw[ultra thick,black] (0,0) circle  (.65);\draw [ultra thick,black,to-] (-.65,0) --(-.65,0);
  }}\endxy 
\\ \label{eq:orinv}
\xy (0,0)*{
\tikzdiagc[yscale=0.8]{
\draw[ultra thick,black,-to] (.5,-.75) -- (.5,.75);
\draw[ultra thick,black,to-] (-.5,-.75) -- (-.5,.75);
}}\endxy\ 
=
\xy (0,0)*{
  \tikzdiagc[yscale=0.8]{
\draw[ultra thick,black,-to] (1,.75) .. controls (1.2,-.05) and (1.8,-.05) .. (2,.75);
\draw[ultra thick,black,to-] (1,-.75) .. controls (1.2,.05) and (1.8,.05) .. (2,.-.75);
            }}\endxy 
          \mspace{80mu}
\xy (0,0)*{
\tikzdiagc[yscale=-0.8]{
\draw[ultra thick,black,-to] (.5,-.75) -- (.5,.75);
\draw[ultra thick,black,to-] (-.5,-.75) -- (-.5,.75);
}}\endxy\ 
=
\xy (0,0)*{
  \tikzdiagc[yscale=-0.8]{
\draw[ultra thick,black,-to] (1,.75) .. controls (1.2,-.05) and (1.8,-.05) .. (2,.75);
\draw[ultra thick,black,to-] (1,-.75) .. controls (1.2,.05) and (1.8,.05) .. (2,.-.75);
}}\endxy
\end{gather}
\endgroup

\item Relation involving oriented strands and distant colored strands:
\begin{equation}\label{eq:orthru4vertex}
\xy (0,0)*{
  \tikzdiagc[yscale=0.5,xscale=.5]{
    \draw[ultra thick,blue] (-1,-1)node[below]{\tiny $i$} -- (.45,.45);
    \draw[ultra thick,myred] (.45,.45) -- (1,1)node[above]{\tiny $i-1$};
    \draw[ultra thick,mygreen] (1,-1)node[below]{\tiny $j$} -- (-.45,.45);
    \draw[ultra thick,olive] (-.45,.45) -- (-1,1)node[above]{\tiny $j-1$};
 \draw[ultra thick,black,to-] (-1,0) ..controls (-.25,.75) and (.25,.75) .. (1,0);
}}\endxy
=
\xy (0,0)*{
  \tikzdiagc[yscale=0.5,xscale=.5]{
    \draw[ultra thick,blue] (-1,-1)node[below]{\tiny $i$} -- (-.45,-.45);
    \draw[ultra thick,myred] (-.45,-.45) -- (1,1)node[above]{\tiny $i-1$};
    \draw[ultra thick,mygreen] (1,-1)node[below]{\tiny $j$} -- (.45,-.45);
    \draw[ultra thick,olive] (.45,-.45) -- (-1,1)node[above]{\tiny $j-1$};
 \draw[ultra thick,black,to-] (-1,0) ..controls (-.25,-.75) and (.25,-.75) .. (1,0);
}}\endxy
\end{equation}

 \item Relations involving oriented strands and two adjacent colored strands:
\begingroup\allowdisplaybreaks
\begin{gather}\label{eq;orReidII}
\xy (0,0)*{
  \tikzdiagc[yscale=2.1,xscale=1.1]{
\draw[ultra thick, blue] (1,0)node[below]{\tiny $i$} .. controls (1,.15) and  (.7,.24) .. (.5,.29);
\draw[ultra thick,myred] (.5,.29) .. controls (-.1,.4) and (-.1,.6) .. (.5,.71);
\draw[ultra thick, blue] (1,1)node[above]{\tiny $i$} .. controls (1,.85) and  (.7,.74) .. (.5,.71) ;
\node[myred] at (-.35,.5) {\tiny $i-1$};
\draw[ultra thick,black,to-] (0,0) ..controls (0,.35) and (1,.25) .. (1,.5) ..controls (1,.75) and (0,.65) .. (0,1);
  }}\endxy
= \ 
\xy (0,0)*{
  \tikzdiagc[yscale=2.1,xscale=1.1]{
\draw[ultra thick, blue] (1,0)node[below]{\tiny $i$} -- (1,1)node[above]{\phantom{\tiny $i$}};
    \draw[ultra thick,black,to-] (0,0) -- (0,1);
  }}\endxy
\mspace{80mu}
\xy (0,0)*{
  \tikzdiagc[yscale=2.1,xscale=-1.1]{
\draw[ultra thick,myred] (1,0)node[below]{\tiny $i-1$} .. controls (1,.15) and  (.7,.24) .. (.5,.29);
\draw[ultra thick,blue] (.5,.29) .. controls (-.1,.4) and (-.1,.6) .. (.5,.71);
\draw[ultra thick,myred] (1,1)node[above]{\tiny $i-1$} .. controls (1,.85) and  (.7,.74) .. (.5,.71) ;
\node[blue] at (-.15,.5) {\tiny $i$};
\draw[ultra thick,black,to-] (0,0) ..controls (0,.35) and (1,.25) .. (1,.5) ..controls (1,.75) and (0,.65) .. (0,1);
  }}\endxy
= \ 
\xy (0,0)*{
  \tikzdiagc[yscale=2.1,xscale=-1.1]{
\draw[ultra thick,myred] (1,0)node[below]{\tiny $i-1$} -- (1,1)node[above]{\phantom{\tiny $i-1$}};
    \draw[ultra thick,black,to-] (0,0) -- (0,1);
  }}\endxy
\\[1ex] \label{eq:dotrhuor}
\xy (0,1.2)*{
\tikzdiagc[scale=1]{
\draw[ultra thick,blue] (.3,-.3)node[below]{\tiny $i$} -- (0,0)node[pos=0, tikzdot]{};
\draw[ultra thick,myred] (-.5,.5)node[above]{\tiny $i-1$} -- (0,0);
\draw[ultra thick,black,to-] (-.5,-.5) -- (.5,.5);
}}\endxy
=
\xy (0,2.5)*{
\tikzdiagc[scale=1]{
\draw[ultra thick,myred] (-.5,.5)node[above]{\tiny $i-1$} -- (-.2,.2)node[pos=1, tikzdot]{};
\draw[ultra thick,black,to-] (-.5,-.5) -- (.5,.5);
}}\endxy
\mspace{80mu}
\xy (0,-.85)*{
\tikzdiagc[xscale=-1,yscale=-1]{
\draw[ultra thick,myred] (.3,-.3)node[above]{\tiny $i-1$} -- (0,0)node[pos=0, tikzdot]{};
\draw[ultra thick,blue] (-.5,.5)node[below]{\tiny $i$} -- (0,0);
\draw[ultra thick,black,-to] (-.5,-.5) -- (.5,.5);
}}\endxy
=
\xy (0,-2.3)*{
\tikzdiagc[xscale=-1,yscale=-1]{
\draw[ultra thick,blue] (-.5,.5)node[below]{\tiny $i$} -- (-.2,.2)node[pos=1, tikzdot]{};
\draw[ultra thick,black,-to] (-.5,-.5) -- (.5,.5);
}}\endxy
\\[1ex] \label{eq:orpitchfork}
\xy (0,0)*{
  \tikzdiagc[yscale=-0.5,xscale=.5]{
    \draw[ultra thick,myred] (0,0)-- (0,.5);
    \draw[ultra thick,myred] (-1,-1)node[above]{\tiny $i-1$} -- (0,0);
    \draw[ultra thick,myred] (1,-1) -- (0,0);
\draw[ultra thick,blue] (0,.5) -- (0,1)node[below]{\tiny $i$};
  \draw[ultra thick,black,to-] (-1,0) ..controls (-.25,.7) and (.25,.7) .. (1,0);
}}\endxy
=
\xy (0,0)*{
  \tikzdiagc[yscale=-0.5,xscale=.5]{
    \draw[ultra thick,myred] (-1,-1)node[above]{\tiny $i-1$} -- (-.45,-.45);
    \draw[ultra thick,myred] (1,-1) -- (.45,-.45);
    \draw[ultra thick,blue] (-.45,-.45)-- (0,0);
    \draw[ultra thick,blue] (.45,-.45)-- (0,0);
  \draw[ultra thick,blue] (0,0)-- (0,1)node[below]{\tiny $i$};
 \draw[ultra thick,black,to-] (-1,0) ..controls (-.25,-.75) and (.25,-.75) .. (1,0);
}}\endxy
\end{gather}
\endgroup

\item Relations involving oriented strands and three adjacent colored strands:
\begin{equation} \label{eq:orslide6vertex}
  \begin{split} 
\xy (0,1)*{
\tikzdiagc[yscale=0.5,xscale=.5]{
  \draw[ultra thick,mygreen] (0,0)-- (0,.6);
  \draw[ultra thick,mygreen] (-1,-1)node[below]{\tiny $i+1$} -- (0,0);
  \draw[ultra thick,mygreen] (1,-1) -- (0,0);
  \draw[ultra thick,blue] (0,.6) -- (0,1);
  \draw[ultra thick,blue] (0,-1)node[below]{\tiny $i$} -- (0,0);
  \draw[ultra thick,blue] (0,0) -- (-.45, .45);
  \draw[ultra thick,myred] (-.45,.45) -- (-1,1)node[above]{\tiny $i-1$};
  \draw[ultra thick,blue] (0,0) -- (.45,.45);
  \draw[ultra thick,myred] (.45,.45) -- (1, 1);
\draw[ultra thick,black,to-] (-1,0) ..controls (-.25,.75) and (.25,.75) .. (1,0);  
}}\endxy
\ &=\  
\xy (0,0)*{
  \tikzdiagc[yscale=0.5,xscale=.5]{
    \draw[ultra thick,mygreen] (-1,-1)node[below]{\tiny $i+1$} -- (-.45,-.45);
    \draw[ultra thick,mygreen] (1,-1) -- (.45,-.45);
    \draw[ultra thick,blue] (0,-1)node[below]{\tiny $i$} -- (0,-.6);
    \draw[ultra thick,myred] (0,-.6) -- (0,0);
  \draw[ultra thick,myred] (0,0) -- (-1,1)node[above]{\tiny $i-1$};
  \draw[ultra thick,myred] (0,0) -- (1, 1);
   \draw[ultra thick,blue] (0,0) -- (0,1);
  \draw[ultra thick,blue] (-.45,-.45) -- (0,0); \draw[ultra thick,blue] (.45,-.45) -- (0,0);
\draw[ultra thick,black,to-] (-1,0) ..controls (-.25,-.75) and (.25,-.75) .. (1,0);  
}}\endxy
\\[1ex] 
\xy (0,1)*{
  \tikzdiagc[yscale=0.5,xscale=.5]{
    \draw[ultra thick,blue] (0,0)-- (0,.6);
    \draw[ultra thick,blue] (-1,-1)node[below]{\tiny $i$} -- (0,0);
    \draw[ultra thick,blue] (1,-1) -- (0,0);
    \draw[ultra thick,myred] (0,.6) -- (0,1)node[above]{\tiny $i-1$};
  \draw[ultra thick,mygreen] (0,-1)node[below]{\tiny $i+1$} -- (0,0);
  \draw[ultra thick,mygreen] (0,0) -- (-.45, .45);  \draw[ultra thick,blue] (-.45,.45) -- (-1,1);
  \draw[ultra thick,mygreen] (0,0) -- (.45,.45);  \draw[ultra thick,blue] (.45,.45) -- (1, 1);
\draw[ultra thick,black,to-] (-1,0) ..controls (-.25,.75) and (.25,.75) .. (1,0);  
}}\endxy
\ &=\  
\xy (0,0)*{
  \tikzdiagc[yscale=0.5,xscale=.5]{
    \draw[ultra thick,blue] (-1,-1)node[below]{\tiny $i$} -- (-.45,-.45);
    \draw[ultra thick,blue] (1,-1) -- (.45,-.45);
  \draw[ultra thick,mygreen] (0,-1)node[below]{\tiny $i+1$} -- (0,-.6); \draw[ultra thick,blue] (0,-.6) -- (0,0);
  \draw[ultra thick,blue] (0,0) -- (-1,1);
  \draw[ultra thick,blue] (0,0) -- (1, 1);
   \draw[ultra thick,myred] (0,0) -- (0,1)node[above]{\tiny $i-1$};
  \draw[ultra thick,myred] (-.45,-.45) -- (0,0); \draw[ultra thick,myred] (.45,-.45) -- (0,0);
\draw[ultra thick,black,to-] (-1,0) ..controls (-.25,-.75) and (.25,-.75) .. (1,0);  
}}\endxy
\end{split}
\end{equation}
\end{itemize}

By relations~\eqref{eq:dotrhuor}, the sum of all colored dumbbells in \eqref{eq:dumbbellsum} 
also commutes with oriented strands, so the corresponding morphism is also central in $\BSext_d$.  

In general, any object in $\BSext_d$ is isomorphic to a direct sum of objects of the form 
$\rB_{\rho}^n\rB_{\underline{w}}$, for some $n\in \mathbb{Z}$ and word $\underline{w}$ 
in $S$. By the relations in~\eqref{eq:orloop}, there is an isomorphism of vector spaces (and of algebras)    
\[
\left(\BSext_d\right)^0 \left(\rB_{\rho}^m, \rB_{\rho}^n\right) \cong 
\begin{cases}
\mathbb{C}\mathrm{id}_{\rB_{\rho}^m},&\text{if}\; m=n;\\
\{0\},&\text{else}.
\end{cases}
\]
Recall that $R=\widehat{\eBS}(\varnothing,\varnothing)$ is the polynomial algebra in the colored dumbbells. Then the isomorphism above generalizes to an isomorphism of graded $R$-$R$-bimodules
\[
\BSext_d \left(\rB_{\rho}^m, \rB_{\rho}^n\right) \cong 
\begin{cases}
R^{\tau^m},&\text{if}\; m=n;\\
\{0\},&\text{else},
\end{cases}
\]
where $\tau$ is the automorphism of $R$ which sends the $i$-colored dumbbell to the $i+1$-colored dumbbell, for any $i\in \widehat{I}$, and $R^{\tau^m}$ is the 
free rank-one $R$-$R$-bimodule with the normal left $R$-action and the right $R$-action 
twisted by $\tau^m$.

Moreover, the black oriented part and the non-oriented colored part of any diagram can be separated by the above relations, resulting in an isomorphism of graded $R$-$R$-bimodules 
\[
\BSext_d\left(\rB_{\rho}^m\rB_{\underline{u}}, \rB_{\rho}^m\rB_{\underline{v}}\right) \cong 
\begin{cases}
R^{\tau^m}\otimes_R \widehat{\eBS}_d\left(\rB_{\underline{u}}, \rB_{\underline{v}}\right), & 
\text{if}\; m=n;\\
\{0\},&\text{else}.
\end{cases}
\]   
In particular, this implies that the natural embedding 
$\widehat{\eBS}_d\hookrightarrow \BSext_d$ is full. For the proofs of these results, 
see~\cite[Section 3.3]{elias2018}. 

\begin{defn}
The {\em extended diagrammatic Soergel category} $\Sext_d$ is the 
idempotent completion of $\left(\BSext_d\right)^{\mathrm{sh}}$. 
\end{defn}

The above results on the Hom-spaces in $\BSext_d$ and Theorem~\ref{thm:categorification} 
imply the following generalization to the extended case, see~\cite[Theorem 2.5]{mackaay-thiel}.

\begin{thm}\label{thm:extcategorification}
For any $n\in \mathbb{Z}$ and $w\in \Saff_d$, the object $\rB_{\rho}^n\rB_w\in \Sext_d$ 
is indecomposable. Moreover, the $\mathbb{Z}[q,q^{-1}]$-linear map 
\begin{gather*}
\left(\eaffh\right)^{\mathbb{Z}[q,q^{-1}]} \to [\Sext_d]_{\oplus}\\
\rho^n b_w\mapsto \rB_{\rho}^n \rB_w, \quad n\in \mathbb{Z}, w \in  \Saff_d
\end{gather*}
is an isomorphism of algebras.
\end{thm}

%
%


\input{sections/Rcomplexes.tex}
\input{sections/evfunct.tex}
\section{Evaluation birepresentations and finitary covers}\label{sec:bireps}

\subsection{Recollections on birepresentation theory}
In the following, we will work with graded (finitary or triangulated) birepresentations of graded, additive bicategories. The particular bicategory we are interested in is, of course, $\Sext_d$, which we view as a bicategory with one object in the usual way.

We call a graded, $\mathbb{C}$-linear, additive category $\mathcal{A}$ 
{\em graded-finitary} if $\mathcal{A}^{\mathrm{sh}}$ is idempotent complete, morphism spaces between indecomposables are finite-dimensional and there are 
only finitely many isomorphism classes of indecomposables up to isomorphism and grading shift. 
Note that $\mathcal{A}$ need not be finitary, because the Hom-spaces might be infinite-dimensional, although they are finite-dimensional in each degree. This is why we write {\em 
graded-finitary} and not {\em graded, finitary}. We denote the $2$-category of graded, resp. graded-finitary, $\mathbb{C}$-linear, additive categories, degree-preserving $\mathbb{C}$-linear functors and natural transformations by $\mathfrak{A}_{\mathbb{C}}^{g}$, resp. $\mathfrak{A}_{\mathbb{C}}^{gf}$. A {\em(locally) graded, additive bicategory} $\cC$ is one whose morphism categories are enriched over $\mathfrak{A}_{\mathbb{C}}^{g}$ and a {\em (locally) graded-finitary bicategory} $\cC$ is one whose 
morphism categories are enriched over $\mathfrak{A}_{\mathbb{C}}^{gf}$ and whose identity $1$-morphisms are indecomposable. Note that, to shorten the string of adjectives, 
we drop the adjective $\mathbb{C}$-linear, even though it is implicit in the enrichment. A {\em graded, additive} (resp. {\em graded-finitary}) {\em birepresentation} is a 
degree-preserving pseudofunctor from $\cC$ to $\mathfrak{A}_{\mathbb{C}}^{g}$ (resp. $\mathfrak{A}_{\mathbb{C}}^{gf}$).

Since we are mainly interested in $\Sext_d$, we will also abuse notation and call additive (bi)cate\-gories of the form $\mathcal{A}^{\mathrm{sh}}$ graded-finitary provided $\mathcal{A}$ is. Similarly, given a graded-finitary birepresentation $\mathbf{M}$ of a graded, additive 
bicategory $\cC$, we will also call the birepresentation $\mathbf{M}^{\mathrm{sh}}$ of 
$\cC^{\mathrm{sh}}$ (which acts on categories $\mathbf{M}(\mathtt{i})^{\mathrm{sh}}$, for objects $\mathtt{i}$, via functors which commute with shifts) graded-finitary. For more detail on these constructions, we refer to \cite[Section 2.6]{mmmtz2019}.

We will also be considering triangulated birepresentations of graded, additive bicategories. Denote by $\mathfrak{T}_\mathbb{C}$ the bicategory of triangulated, $\mathbb{C}$-linear categories, ($\mathbb{C}$-linear) triangulated functors and natural transformations. A {\em triangulated birepresentation} of a $\mathbb{C}$-linear, additive bicategory $\cC$ is a 
($\mathbb{C}$-linear) pseudofunctor from $\cC$ to $\mathfrak{T}_\mathbb{C}^{gf}$. In order to consider graded versions, we restrict ourselves to the $2$-full subbicategory $\mathfrak{T}_\mathbb{C}^{g}$ of $\mathfrak{T}_\mathbb{C}$ whose objects are triangulated categories of the form $\kb(\mathcal{A}^{\mathrm{sh}})$ for a graded, $\mathbb{C}$-linear, additive category $\mathcal{A}$, and whose functors are degree-preserving triangulated functors. 
 A {\em graded-triangulated birepresentation} of an additive, graded bicategory $\cC$ is then a degree-preserving ($\mathbb{C}$-linear) pseudofunctor from $\cC$ to $\mathfrak{T}_\mathbb{C}^{g}$.

Similarly to the finitary case above, we will call a birepresentation graded, triangulated if 
a bicategory of the form $\cC^{\mathrm{sh}}$ acts on triangulated categories of the form $\mathcal{T}^{\mathrm{sh}}$ via triangulated functors commuting with shifts. These are birepresentations obtained by taking a graded birepresentation of $\cC$ acting on $\mathcal{T}$, closing under shifts, and then restricting to morphisms of degree zero. 

In some cases, graded-finitary birepresentations will have an additional shift functor (coming from the homological shift in a triangulated birepresentation), with respect to which morphisms in the underlying categories will have degree zero. We call such birepresentations bigraded-finitary.

Given a (locally) additive, graded bicategory, the set of isomorphism classes of indecomposable $1$-morphisms up to grading shift can be given three natural partial preorders: the {\em left} 
preorder ($[\mathrm{F} ]\leq_L [\mathrm{G}]$ if and only if $[\mathrm{G}]$ appears as a direct summand of $[\mathrm{HF}]$ for some $1$-morphism $\mathrm{H}$), the {\em right} preorder 
($[\mathrm{F} ]\leq_R [\mathrm{G}]$ if and only if $[\mathrm{G}]$ appears as a direct summand of $[\mathrm{FH}]$ for some $1$-morphism $\mathrm{H}$) and the {\em two-sided} preorder 
($[\mathrm{F} ]\leq_J [\mathrm{G}]$ if and only if $[\mathrm{G}]$ appears as a direct summand of $[\mathrm{H_1FH_2}]$ for some $1$-morphisms $\mathrm{H_1}, \mathrm{H_2}$), and the corresponding equivalence classes are called {\em left, right} and 
{\em two-sided cells}, respectively.

If $\cC$ is graded-finitary, we can associate to any left cell a so-called {\em graded cell $2$-representation}, which is the quotient of the left $2$-ideal in $\cC$ generated by the identities on the $1$-morphisms in the cell, by the unique maximal ideal of the resulting birepresentation (i.e. the unique maximal ideal of the underlying categories which is stable under the action of $\cC$). For more details (in the ungraded case, but the graded one is analogous), 
see e.g. \cite[Section 3.3]{MM5}.

\subsection{Finitary covers of evaluation cell birepresentations}

Let $\mathbf{M}$ be a graded-finitary birepresentation of $\eS_d$, for any 
$d\in \mathbb{N}_{\geq 2}$. Then $\kb(\mathbf{M})$, as a graded, triangulated 
birepresentation of $\kb(\eS_d)$, induces a graded, triangulated birepresentation of 
$\Sext_d$, the {\em evaluation birepresentation} $\mathbf{M}^{\Ev_{r,s}}$, resp. $\mathbf{M}^{\Ev'_{r,s}}$, by pull-back through the evaluation functors $\Ev_{r,s}$, resp. $\Ev'_{r,s}$, for any $r,s\in \mathbb{Z}$. 

In this subsection we show that, if $\mathbf{M}$ is a cell birepresentation of $\eS_d$, then $\mathbf{M}^{\Ev_{r,s}}$ has a bigraded-finitary cover in the following sense. 

\begin{defn}\label{defn:fincov}
A {\em bigraded-finitary cover} of a graded, triangulated birepresentation 
$\mathbf{N}$ of a graded, additive bicategory $\mathcal{C}$ is a bigraded-finitary 
birepresentation $\mathbf{L}$ of $\mathcal{C}$ 
together with a faithful morphism of linear additive bigraded birepresentations 
$\Phi\colon \mathbf{L}\to \mathbf{N}$ whose essential image generates $\mathbf{N}$ 
as a graded triangulated category.
\end{defn}

\begin{prop}\label{prop:fincovgen}
Let $\mathbf{M}$ be the graded cell birepresentation associated to some left 
cell $\mathcal{L}$ of $\eS_d$. Then $\mathbf{M}^{\Ev_{r,s}}$ has a 
bigraded-finitary cover.
\end{prop}

\begin{proof} By \cite[Proposition 4.31]{elias-hogancamp}, $\rT_{\rho}^d$ 
acts as $\mathrm{Id}\langle x\rangle [y]$ on $\mathbf{M}^{\Ev_{r,s}}$, for some $x,y\in \mathbb{Z}$. 
Let $\mathbf{L}$ be the closure under isomorphisms, direct sums, 
direct summands, grading and homological shifts of the 
$\Ev_{r,s}(\rT_{\rho}^i) \rB_w$, for $i\in \widehat{I}$ and $w\in \mathcal{L}$. 
Relation~\eqref{eq;orReidII} 
implies that $\mathbf{L}$ is a bigraded-finitary birepresentation of $\Sext_d$. 

The inclusion functor $\mathbf{L}\hookrightarrow \mathbf{M}^{\Ev_{r,s}}$ is a morphism of linear additive bigraded birepresentations and its essential image 
generates $\kb(\mathbf{M})$, which is the underlying graded triangulated category of 
$\mathbf{M}^{\Ev_{r,s}}$.   
\end{proof}

We refer to Corollary~\ref{cor:fincover} for an example demonstrating that $\Phi$ is not full in general. 

\begin{rem} It is easy to see that $\mathbf{L}$ is transitive, and it looks likely that calculations, using the explicit descriptions of the representing bimodules for the $\rB_w$ given in \cite[Section 4.3]{mmmtz2019}, one can verify that it is indeed simple transitive.
\end{rem}


\subsection{The zigzag algebras}

Let us first recall the {\em affine zigzag algebra} $\widehat{Z}_d$ over $\mathbb{C}$ associated to the $\widehat{A}_{d-1}$ Dynkin diagram. As is well-known, there are two isomorphism classes of 
affine zigzag algebras with invertible integer coefficients, and we use a specific 
representative of either one or the other depending on the parity of $d$. 

Let $e_i, \; i\in \widehat{I},$ denote the orthogonal idempotents associated to 
the vertices of the zigzag quiver 

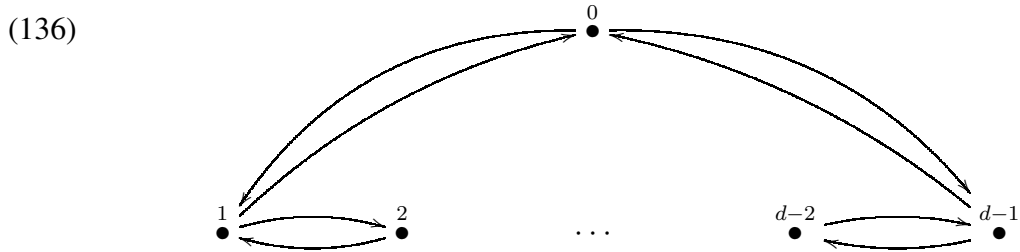
\begin{equation}
\xymatrix{
&&&&\overset{0}{\bullet}\ar@/_2pc/[ddllll]\ar@/^2pc/[ddrrrr]&&&&\\\\
\overset{1 }{\bullet}\ar@/^0.5pc/[rr]\ar@/^1pc/[uurrrr]&& \overset{2}{\bullet}\ar@/^0.5pc/[ll]&& \cdots && \overset{d-2}{\bullet} \ar@/^0.5pc/[rr]&& \overset{d-1}{\bullet}\ar@/^0.5pc/[ll]\ar@/_1pc/[uullll]
}
\end{equation}

and $i_1\vert i_2\vert \ldots \vert i_k$ the path in the quiver 
from $i_k$ to $i_1$ via $i_{k-1},\ldots, i_2$. The relations in $\widehat{Z}_d$ are 
\begin{gather}
i\vert i+1 \vert i+2=0=i\vert i-1\vert i-2, \quad i\in \widehat{I}; \\ 
i\vert i+1\vert i = i\vert i-1\vert i, \quad i \in I; \\
0 \vert 1 \vert 0 = (-1)^d (0 \vert d-1 \vert 0). 
\end{gather}
For convenience, we also use the notation 
\[
\ell_i:= i\vert i+1\vert i,
\]
for any $i \in \widehat{I}$. This algebra has dimension $4d$, it is positively graded by putting the degree of every path 
equal to its length, and it is a graded Frobenius algebra 
with non-degenerate trace defined by 
\begin{equation}
\mathrm{tr}(\ell_i)=1\;\text{for every}\;i \in \widehat{I};\quad \mathrm{tr}(a)=0\;\text{when} \deg(a) \ne 2.
\end{equation}
This means that $\widehat{Z}_d^\star \cong \widehat{Z}_d\langle 2\rangle$ as graded left, resp. right, $\widehat{Z}_d$-modules. Define the non-degenerate bilinear pairing 
$\langle .,.\rangle\colon \widehat{Z}_d\otimes \widehat{Z}_d\to \mathbb{C}$ as usual 
\begin{equation}
\langle a,b\rangle :=\mathrm{tr}(ab),\; a,b\in \widehat{Z}_d,
\end{equation}
and recall that two bases of $\widehat{Z}_d$, say $\{a_i\mid i=1\ldots, 4d\}$ and 
$\{a_i^{\star}\mid 1,\ldots,4d\}$, are called {\em dual} to each other if they satisfy
\[
\langle a_i,a_j^{\star}\rangle =\delta_{i,j},\; i,j=1,\ldots,4d,
\]
where $\delta_{i,j}$ is the Kronecker delta. With respect to the bilinear form on 
$\widehat{Z}_d$,  there is a natural pair of dual bases 
$\{e_i,\ell_i ,(i\vert i\pm 1)\}_{ i\in \widehat{I}}$ and 
$\{e_i^{\star},\ell_i^{\star},(i\vert i\pm 1)^{\star}\}_{i\in \widehat{I}}$, such that  
\begin{gather}
e_i^\star= \ell_i,\; \ell_i^\star = e_i, \quad i \in \widehat{I}; \\
(i \vert (i\pm 1))^\star = (i\pm 1) \vert i, \quad i\in I; \\
 (0\vert (d-1))^{\star}=(-1)^d ((d-1)\vert 0).
\end{gather}
Note that $\widehat{Z}_d$ is symmetric when $d$ is even and only weakly symmetric when $d$ is odd. 

Let $\widehat{Z}_d{-}\mathrm{fgproj}$, resp. $\mathrm{fgproj}{-}\widehat{Z}_d$,  
be the category of finite-dimensional, graded, projective left, resp. right, 
$\widehat{Z}_d$-modules and degree-preserving module maps. The indecomposable 
objects in these categories are isomorphic to $\widehat{Z}_de_i\langle t\rangle$, resp. 
$e_i\widehat{Z}_d\langle t\rangle,$ for some $i\in \widehat{I}$ and $t\in \mathbb{Z}$.

Finally, let $\widehat{Z}_d{-}\mathrm{fgbiproj}{-}\widehat{Z}_d$ be the monoidal category of all finite-dimensional, graded, biprojective $\widehat{Z}_d{-}\widehat{Z}_d$-bimodules and degree-preserving bimodule maps. A bimodule is called biprojective if it is projective as 
a graded left module and as a graded right module, but not necessarily as a graded bimodule. 
Every indecomposable projective object in this category is isomorphic to
\[
\widehat{Z}_d e_i\otimes e_j\widehat{Z}_d\langle t \rangle,
\]
for some $i,j \in \widehat{I}$ and $t\in \mathbb{Z}$. The monoidal structure of 
$\widehat{Z}_d{-}\mathrm{fgbiproj}{-}\widehat{Z}_d$ is given by 
tensoring over $\widehat{Z}_d$ and the unit object is $\widehat{Z}_d$, which is biprojective 
but not projective as a bimodule over itself. Recall that any exact, graded 
endofunctor of $\widehat{Z}_d{-}\mathrm{fgproj}$ is 
naturally isomorphic to $B \otimes_{\widehat{Z}_d} -$, for some $B\in \widehat{Z}_d{-}\mathrm{fgbiproj}{-}\widehat{Z}_d$. Natural transformations between exact, graded 
endofunctors correspond to $\widehat{Z}_d{-}\widehat{Z}_d$-bimodule maps 
and the composition of endofunctors corresponds to the tensor product of the corresponding bimodules over $\widehat{Z}_d$. 

Let $\tau$ be the degree-preserving algebra automorphism of $\widehat{Z}_d$ induced by the counterclockwise rotation of the Dynkin diagram defined by  
\begin{equation}
e_i \mapsto e_{i+1},\quad 0\vert (d-1) \mapsto (-1)^d (1\vert 0),\quad i\vert j \mapsto (i+1) 
\vert (j+1), 
\end{equation}
for $i,j\in \widehat{I}$, such that $j=i\pm 1$ but $(i,j)\ne (0,d-1)$. Note that $\tau^{d}=\mathrm{id}$ when $d$ is even, and $(\tau)^{2d}=\mathrm{id}$ when $d$ is odd. By definition, the {\em twisted bimodule} 
\begin{equation}
\widehat{Z}^{\tau}_n\in \widehat{Z}_d{-}\mathrm{fgbiproj}{-}\widehat{Z}_d
\end{equation}
has underlying vector space $\widehat{Z}_d$, while the left and right 
$\widehat{Z}_d$-actions are defined by 
\begin{equation}
a\cdot_L b \cdot_R c:=ab\tau(c),
\end{equation}
for $a,b,c\in \widehat{Z}_d$. It is clear that $\widehat{Z}^{\tau}_d\cong \widehat{Z}_d$ 
as left and as right $\widehat{Z}_d$-modules, but not as 
$\widehat{Z}_d$-$\widehat{Z}_d$-bimodules.
As a consequence, $\widehat{Z}^{\tau}_d$ is
biprojective. It is, however, 
not projective 
as a $\widehat{Z}_d$-$\widehat{Z}_d$-bimodule. We record the existence of an isomorphism 
\begin{gather}\label{eq:lefttau1}
\widehat{Z}^{\tau^k}_d\otimes_{\widehat{Z}_d} \widehat{Z}^{\tau^m}_d \cong \widehat{Z}^{\tau^{k+m}}_d
\end{gather}
in $\widehat{Z}_d{-}\mathrm{fgbiproj}{-}\widehat{Z}_d$, for every pair $k,m\in \mathbb{Z}$.  

Note further that there exist isomorphisms of left, resp. right, 
$\widehat{Z}_d$-modules 
\begin{equation}\label{eq:lefttau2}
 \widehat{Z}_d^{\tau}\otimes_{\widehat{Z}_d} \widehat{Z}_d e_i  \cong \widehat{Z}_d e_{i+1}
\quad\text{and}\quad e_i \widehat{Z}_d \otimes_{\widehat{Z}_d} \widehat{Z}_d^{\tau}\cong e_{i-1} \widehat{Z}_d
\end{equation}
and, therefore, an isomorphism of $\widehat{Z}_d$-$\widehat{Z}_d$-bimodules
\begin{equation}\label{eq:lefttau3}
\widehat{Z}_d^{\tau}\otimes_{\widehat{Z}_d} \widehat{Z}_d e_i \otimes e_i\widehat{Z}_d  
\cong \widehat{Z}_d e_{i+1} \otimes e_{i+1}\widehat{Z}_d \otimes_{\widehat{Z}_d}\widehat{Z}_d^{\tau}
\end{equation}
for every $i\in \widehat{I}$.

The {\em zigzag algebra} $Z_d$ of finite type $A_{d-1}$ is by definition the idempotent
subalgebra 
\begin{equation}
(e_1+\cdots+e_{d-1})\widehat{Z}_d(e_1+\cdots+e_{d-1}).
\end{equation}

\subsection{The birepresentations}

Let $Z=Z_d$ denote the zigzag algebra of finite type $A_{d-1}$. Recall the finitary birepresentation $\mathbf{M}_d$ of $\eS_d$ acting on $Z$-$\mathrm{gproj}$, 
the finitary category of finite-dimensional, graded projective $Z$-modules, by graded, biprojective 
$Z$-$Z$-bimodules. Under this birepresentation, $\mathbbm{1}=R$ acts by tensoring (over $Z$) with $Z$ and each $\rB_i$ acts by tensoring (over $Z$) with $Ze_i\otimes e_iZ\langle 1\rangle$, for $i\in I$. The image of the generating Soergel diagrams is given by 
\begin{gather}
\begin{array}{lcrcl}
\mathbf{M}_d\biggl(
\xy (0,-4)*{
\tikzdiagc[scale=1]{
\begin{scope}[yscale=-.5,xscale=.5,shift={(5,-2)}] 
  \draw[ultra thick,blue] (-1,0) -- (-1, 1)node[pos=0, tikzdot]{} node[below]{\tiny $i$};
\end{scope}
}}
\endxy
\biggr)
& \colon & Z e_i \otimes e_i Z\langle 1 \rangle & \to  &Z \\
&& ae_i \otimes e_ib & \mapsto  & ae_ib,  
\end{array}
\\[2ex]
\begin{array}{lcrcl}
\mathbf{M}_d\biggl(
\xy (0,4)*{
\tikzdiagc[scale=1]{
\begin{scope}[yscale=-.5,xscale=.5,shift={(5,-2)}] 
  \draw[ultra thick,blue] (-1,0) -- (-1, 1)node[pos=1, tikzdot]{} node[above, yshift=14pt]{\tiny $i$};
\end{scope}
}}
\endxy
\biggr)
& \colon &  Z  & \to  & Z e_i \otimes e_i Z\langle 1 \rangle \\
&& e_j & \mapsto  &
\begin{cases}
(-1)^i \left(\ell_i \otimes e_i + e_i \otimes \ell_i\right), & j=i;\\
(-1)^i\left(j\vert i \otimes i \vert j\right), & j\pm 1=i,
\end{cases}
\end{array}
\\[2ex] 
\begin{array}{lcrcl}
\mathbf{M}_d\biggl(\xy (0,0)*{
\tikzdiagc[scale=1]{
\begin{scope}[yscale=.5,xscale=.5,shift={(8,2)}] 
   \draw[ultra thick,blue] (0,0)-- (0, 1) node[above]{\tiny $i$}; 
  \draw[ultra thick,blue] (-1,-1) -- (0,0) node[below, shift={(-0.5,-0.5)}]{\tiny $i$}; 
\draw[ultra thick,blue] (1,-1) -- (0,0) node[below, shift={(0.5,-0.5)}]{\tiny $i$};
\end{scope}
}}
\endxy
\biggr)
&\colon & Z e_i \otimes e_i  Z e_i \otimes e_i  Z \langle 2 \rangle & \to 
&  Z e_i \otimes e_i  Z \langle 1 \rangle \\
&& e_i \otimes e_i ae_i \otimes e_i  &\mapsto & (-1)^i \mathrm{tr}(e_i a e_i) e_i \otimes e_i 
\end{array}
\\[2ex]
\begin{array}{lcrcl}
\mathbf{M}_d\biggl(\xy (0,0)*{
\tikzdiagc[scale=1]{
\begin{scope}[yscale=.5,xscale=.5,shift={(8,2)}] 
   \draw[ultra thick,blue] (0,0)-- (0, 1) node[below, yshift=-14pt]{\tiny $i$}; 
  \draw[ultra thick,blue] (0,1) -- (-1,2) node[above]{\tiny $i$}; 
\draw[ultra thick,blue] (0,1) -- (1,2) node[above]{\tiny $i$};
\end{scope}
}}
\endxy
\biggr)
&\colon &  Z e_i \otimes e_i  Z \langle 1 \rangle 
&\to &  Z e_i \otimes e_i  Z e_i \otimes e_i  Z \langle 2 \rangle \\
&& e_i \otimes e_i &\mapsto & e_i \otimes e_i \otimes e_i,
\end{array}
\end{gather}
while all other generating Soergel diagrams are sent to zero. The proof that this is well-defined 
is a straightforward computation and similar to the proof of 
\cite[Theorem I]{mackaay-tubbenhauer}. It is easy to see that this birepresentation decategorifies 
to the representation $M_d$ of $H_d$, given in \eqref{eq:M}.

Now, consider the triangulated birepresentations $\mathbf{M}^{\Ev_{r,s}}$ and 
$\mathbf{M}^{\Ev'_{-r,-s}}$ of $\Sext_d$, for $r,s\in \mathbb{Z}$, obtained by 
pulling $\kb(\mathbf{M})$ back through the evaluation functors $\Ev_{r,s}$ and $\Ev'_{-r,-s}$. 
These decategorify to $M^{\ev_{a}}$ and $M^{\ev_{a^{-1}}}$ 
defined in \eqref{eq:action-rho} and \eqref{eq:action-prime-rho}, respectively, where 
$a=(-1)^s q^r$. The case $(r,s)=(d-2, 2-d)$ is somewhat special, as it corresponds to the so-called 
{\em Tate twist}, but the general case can easily be derived from this one by shifting the 
bigrading in all arguments below. To keep the notation simple, we therefore consider 
$\mathbf{M}^{\Ev_{r,s}}$ for the fixed choice $(r,s)=(d-2,2-d)$ first.

Define the complex
\begin{equation}
X_0:= \bigl(
\underline{Ze_{d-1}\langle 1\rangle} \to Ze_{d-2}\langle 2\rangle \to\cdots \to 
Ze_1\langle d-1\rangle
\bigr)
\end{equation}
where the term $Ze_{d-1}\langle 1\rangle$ is in homological degree $0$ and the differential in position $i$ is given by right multiplication by $d-i-1\vert d-i-2$. We further 
set $X_i:=Ze_i$, for $i\in I$. 

In Proposition~\ref{prop:fincovgen}, the rank of the bigraded-finitary cover $\mathbf{L}$ 
of an evaluation cell birepresentation is not necessarily minimal. In the following proposition, 
we give a minimal finitary cover for $\mathbf{M}^{\Ev_{r,s}}$.

\begin{prop}\label{prop:invariant}
The bigraded-finitary subcategory 
\[
\widehat{\mathbf{M}}_{d-2,2-d}:=
\mathrm{add}\left\{(X_0\oplus X_1\oplus\cdots \oplus X_{d-1})
\langle i\rangle[j]\mid i,j\in \mathbb{Z} \right\}
\]
is stable under the action of $\Sext_d$, and hence carries the structure of a 
finitary birepresentation of $\Sext_d$, which we denote by the same symbol. 
\end{prop}
\begin{proof}
We need to check stability under $\rB_1\ldots, \rB_{d-1}$ and $\rT_{\rho}$. The action of $\rB_1\ldots, \rB_{d-1}$ stabilises $\mathrm{add}\left\{X_1\oplus\cdots \oplus X_{d-1}\langle i\rangle[j]\mid i,j\in \mathbb{Z} \right\}$ since this is just the finitary birepresentation of $\eS_d$ described above. We therefore first compute $\rB_i(X_0)$ for $i\in I$ and then verify stability of $\mathrm{add}\left\{X_1\oplus\cdots \oplus X_{d-1}\langle i\rangle[j]\mid i,j\in \mathbb{Z} \right\}$ under $\rT_{\rho}$.

Notice that, for $i\in \{2,\cdots d-2\}$, $\rB_i(X_0)$ is given by the complex
$$Ze_i\otimes \biggl( e_iZe_{i+1}\langle d-i \rangle \to e_iZe_{i}\langle d-i +1 \rangle \to e_iZe_{i-1}\langle d-i+2 \rangle \biggr). $$
Notice that, as a $Z$-module, this is just $Ze_i$ tensored with a complex of vector spaces, so it suffices to argue that said complex of vector spaces is null-homotopic. This is indeed the case since the first map embeds a one-dimensional space into a two-dimensional space, and the second map is a surjection onto another one-dimensional space. It follows that the whole complex is null-homotopic.

Further, $\rB_1(X_0)$ is given by
$$Ze_1\otimes \biggl( e_1Ze_2\langle d-1\rangle \to  e_1Ze_1\langle d\rangle \biggr)$$
with map $1|2\mapsto \ell_1$, which is injective, hence the summand surviving Gaussian elimination is $Ze_1\otimes e_1\langle d\rangle $ in homological degree $d-2$. Thus 
the result is homotopy equivalent to $Ze_1\langle d\rangle [2-d]= X_1\langle d \rangle [2-d]$.

On the other extreme, $\rB_{d-1}(X_0)$ is given by
$$Ze_{d-1}\otimes \biggl( \underline{e_{d-1}Ze_{d-1}\langle 2\rangle }\to  e_{d-1}Ze_{d-2}\langle 3\rangle \biggr)$$
where the map is right multiplication by $d-1|d-2$, which is surjective. The kernel is thus $Ze_{d-1}\otimes \ell_{d-1}\langle 2\rangle$ and the result is homotopy equivalent to $Ze_{d-1}= X_{d-1}$ without any shifts. 

Thus $\mathrm{add}\left\{X_0\oplus X_1\oplus\cdots \oplus X_{d-1}\langle i\rangle[j]\mid i,j\in \mathbb{Z} \right\}$ is stable under the action of $\rB_1,\ldots,\rB_{d-1}$.

It remains to show that $\widehat{\mathbf{M}}$ is stable under the action of 
$\rT_{\rho}$. Recall from Section~\ref{sec:defevalfunctor}
that $\Ev_{d-2,2-d}(\rT_{\rho})=
\rT^{-1}_{1}\cdots \rT^{-1}_{d-1}\langle d-2\rangle [2-d]$ and 
\begin{equation*}
  \rT_i^{-1} = R\langle -1\rangle \xra{
\mspace{15mu}
    \xy (0,0)*{
  \tikzdiagc[yscale=.3,xscale=.25]{
  \draw[ultra thick,blue] (-1,0) -- (-1, 1)node[pos=0, tikzdot]{};
}}\endxy
\mspace{15mu}
}\underline{\Bi}. 
\end{equation*} 

Using the definition of $\mathbf{M}_d$ above, it is easy to see that the complex representing 
$\Ev_{d-2,2-d}(\rT_{\rho})$ is

\begin{equation}\label{eq:imageTrho}
\resizebox{\textwidth}{!}{
\xymatrix@R=3pt
{
&\underline{ Ze_1\otimes e_1Z \langle 1 \rangle} \ar[dr]&&&&\\
&&Ze_1\otimes e_2Z\langle 2 \rangle \ar[dr]&&&\\
&\underline{ Ze_2\otimes e_2Z\langle 1 \rangle} \ar[ur]\ar[dr] &&\ddots\ar[dr]&&\\
&&\ddots&&Ze_1\otimes e_{d-2}Z\langle d-2 \rangle \ar[dr]&\\
Z\langle -1\rangle \ar[uuuur]\ar[uur]\ar[ddddr]\ar[ddr]&&&\ar[dr]\ar[ur]&&
Ze_1\otimes e_{d-1}Z\langle d-1 \rangle \\
&&\iddots&&Ze_2\otimes e_{d-1}Z\langle d-2 \rangle \ar[ur]&\\
&\underline{Ze_{d-2}\otimes e_{d-2}Z\langle 1 \rangle }\ar[dr]\ar[ur]&&\ar[ur]\iddots&&\\
&&Ze_{d-2}\otimes e_{d-1}Z\langle 2 \rangle \ar[ur]&&&\\
&\underline{Ze_{d-1}\otimes e_{d-1}Z\langle 1 \rangle }\ar[ur]&&&&
} }
\end{equation}
Here, in the first differential, whose source is $Z\langle -1\rangle$, the component mapping to $Ze_i\otimes e_iZ \langle 1 \rangle$
is given by 
\[
e_j \mapsto  
\begin{cases}
\ell_i \otimes e_i + e_i \otimes \ell_i, &\mathrm{if}\, i=j;\\  
j\vert i \otimes i \vert j, &\mathrm{if}\, i\ne j.
\end{cases}
\] 
The other differentials are all vectors of $Z$-$Z$-bimodule maps which are equal to the tensor 
product of $\pm \mathrm{id}$ on one tensor factor and $i\vert i+1$, for some $i=1,\ldots, d-2$, 
on the other tensor factor. For our arguments below, the signs of these maps are not important.

We are first going to prove that $\Ev_{d-2,2-d}(\rT_{\rho})(X_i)\simeq X_{i+1}$, for any $i=1,\ldots, d-2$. Since 
$e_jZe_i=\{0\}$ when $\vert i-j\vert >1$, the non-zero part of the complex corresponding to 
$\Ev_{d-2,2-d}(\rT_{\rho})(X_i)$ is

\begin{equation*}
\resizebox{\textwidth}{!}{
\xymatrix
{
&\underline{ Ze_{i-1}\otimes e_{i-1}Ze_i\langle 1 \rangle }\ar[r] \ar[dr]& 
Ze_{i-2}\otimes e_{i-1}Ze_i\langle 2 \rangle  \ar[r]\ar[dr]&\cdots&
Ze_1\otimes e_{i-1}Ze_i \langle i-1 \rangle \ar[dr]&&\\
Ze_i\langle -1 \rangle \ar[ur]\ar[r]\ar[dr]&\underline{ Ze_{i}\otimes e_{i}Ze_i\langle 1 \rangle }\ar[r]\ar[dr]& Ze_{i-1}\otimes e_{i}Ze_i\langle 2 \rangle  \ar[r]\ar[dr]&\cdots&
Ze_2\otimes e_{i}Ze_i\langle i-1 \rangle \ar[r]\ar[dr]&Ze_1 \otimes e_{i}Ze_i\langle i 
\rangle \ar[dr]&\\
& {\color{purple}\underline{Ze_{i+1}\otimes e_{i+1}Ze_i\langle 1 \rangle }}\ar[r]& 
Ze_{i}\otimes e_{i+1}Ze_i\langle 2 \rangle \ar[r]&\cdots&Ze_3\otimes e_{i+1}Ze_i
\langle i-1 \rangle \ar[r]&Ze_2\otimes e_{i+1}Ze_i \langle i \rangle \ar[r]&
Ze_1\otimes e_{i+1}Ze_i\langle i+1 \rangle \\
} }
\end{equation*}
By Gaussian elimination, one can then see that this is homotopy equivalent to the purple 
$Ze_{i+1}\otimes e_{i+1}Ze_i\langle 1 \rangle$ in homological degree zero, which is isomorphic to $X_{i+1}$. 
To explain this, we identity each vertex of the diagram above by its pair of coordinates (row number, column number), 
where we number the rows of the complex by 1,2,3 from top to bottom and 
the columns by their homological degree. As in the diagram above, we omit the signs of all maps below, 
since they are not important for our argument. Using these conventions, first note that the the part of the complex $(2,-1)\to (2,0)\to (3,1)$ is given by 
$$Ze_i \otimes \biggl(\Bbbk\langle -1\rangle \to \underline{e_iZe_i\langle 1\rangle} \to e_{i+1}Ze_i\langle 2\rangle\biggr)$$
where the complex of vector spaces is split by the same arguments as above and hence null-homotopic. Thus these
three terms cancel in the Gaussian elimination procedure. Similarly, every part of the complex of the form 
$(1,j)\to (2,j+1)\to (3,j+2)$, for $j=0,\ldots, i-2$, is given by 
$$Ze_{i-j-1}\otimes \biggl(e_{i-1}Ze_i\langle j+1\rangle \to  e_iZe_i \langle j+2\rangle\to e_{i+1}Ze_i\langle j+3\rangle   \biggr) $$
is split and hence null-homotopic.
Hence all these triples of terms cancel in the Gaussian elimination procedure, which in the end only leaves the purple one, proving the desired homotopy equivalence.

The next homotopy equivalence we are going to prove is $\Ev_{d-2,2-d}(\rT_{\rho})(X_{d-1})\simeq X_0$. The non-zero part of the complex $\Ev_{d-2,2-d}(\rT_{\rho})(X_{d-1})$ is 

\begin{equation*}
\resizebox{\textwidth}{!}{
\xymatrix
{
& \underline{Ze_{d-2} \otimes e_{d-2}Ze_{d-1}\langle 1\rangle}\ar[r]\ar[ddr] & Ze_{d-3} \otimes e_{d-2}Ze_{d-1}\langle 2\rangle 
\ar[r]\ar[ddr]&\cdots \ar[r]\ar[ddr] & Ze_1\otimes e_{d-2}Ze_{d-1}\langle d-2 \rangle \ar[dr] &  \\
Ze_{d-1}\langle -1\rangle \ar[ur] \ar[dr]&&&&& {\color{purple}Ze_1 \otimes e_{d-1}Ze_{d-1}\langle d-1 \rangle }\\
& {\color{purple}\underline{Ze_{d-1} \otimes e_{d-1}Z e_{d-1} \langle 1\rangle}} \ar[r] &  
{\color{purple}Ze_{d-2} \otimes e_{d-1}Ze_{d-1}\langle 2 \rangle} \ar[r] & \cdots \ar[r]& 
{\color{purple}Ze_2 \otimes e_{d-1}Z e_{d-1}\langle d-2\rangle}  \ar[ur] & \\   
} }
\end{equation*}
where the differentials are as above. Note that again all maps pointing to the south-east in the complex are given by the tensor product of the identity of some $Ze_{d-j}$ with an injective map of vector spaces hence split. Thus, by Gaussian elimination, this complex is homotopy equivalent to the direct summand 
of the purple subcomplex for which the right tensor factor is restricted to multiples of $e_{d-1}$. This direct summand is 
indeed isomorphic to $X_0$. 

The remaining case of the action of $\Ev_{d-2,2-d}(\rT_{\rho})$ on $X_0$ can be replaced by considering the action of $\Ev_{d-2,2-d}(\rT_{\rho})^{-1}$ on $X_1$, which is analogous to the action of $\Ev_{d-2,2-d}(\rT_{\rho})$ on $X_{d-2}$.
\end{proof}

Similarly, we can define an additive birepresentation 
$\widehat{\mathbf{M}}_{r,s}$ of $\Sext_d$, for any $r,s\in \mathbb{Z}$. 

\begin{cor}\label{cor:fincover}
For any $r,s\in \mathbb{Z}$, there is a morphism of additive $\Sext_d$-birepresentations 
$\Phi\colon \widehat{\mathbf{M}}_{r,s}\to \mathbf{M}^{\Ev_{r,s}}$, induced by the embedding from Proposition~\ref{prop:invariant}. This makes $\widehat{\mathbf{M}}_{r,s}$ into a finitary cover of $\mathbf{M}^{\Ev_{r,s}}$.   
\end{cor}
\begin{proof}
All assertions follow immediately from Proposition~\ref{prop:invariant}.
\end{proof}

\begin{rem}
Note that $\widehat{\mathbf{M}}_{r,s}$ decategorifies to the Graham--Lehrer cell module 
$\widehat{M}_{d,\lambda}$ with $\lambda= (-1)^{s-(2-d)}q^{r-(d-2)}$, as can be 
easily seen by comparing the action of the generators on the $X_i$ with the decategorified action in \eqref{eq:GL} and \eqref{eq:GLext}. Moreover, $\Phi$ decategorifies to the projection of $\widehat{M}_{r,s}$ onto $L^{+}_{d,(-1)^{s-(2-d)}q^{r-(d-2)}}$.
\end{rem}

\begin{prop}\label{prop:isoalgs}
For any $r,s\in \mathbb{Z}$, there is an isomorphism of ungraded algebras 
\[
\mathrm{End}_{ \mathbf{M}^{\Ev_{r,s}}}(X_0\oplus\cdots\oplus X_{d-1})\cong \widehat{Z}.
\]  
\end{prop}

\begin{proof}
Without loss of generality, we assume that $(r,s)=(d-2,2-d)$, as before. Denote by $p_{d-1}\colon X_0 \to X_{d-1}$ the projection onto the component in homological degree $0$ and by $j_d\colon X_{d-1} \to X_0$ the map induced by multiplication with $\ell_{d-1}$.  Similarly, denote by $j_1\colon X_1[2-d] \to X_0$ the inclusion of the component in homological degree $d-2$ and by $p_1\colon X_0  \to X_1[2-d]$ the map induced by multiplication with $\ell_1$. We remark that $p_{d-1},j_{d-1}, j_1,p_1$ have degrees $1, 1, 1-d, d+1$, respectively. 
Moreover, we denote the maps $Ze_i \to Ze_{i\pm 1}$ given by right multiplication by $i\vert i\pm 1$ by $r_{i\vert i\pm 1}$. Then it is a straightforward calculation to verify that  $\mathrm{End}_{ \mathbf{M}^{\Ev_a}}(X_0\oplus\cdots\oplus X_{d-1})$ is given by the path algebra of the quiver
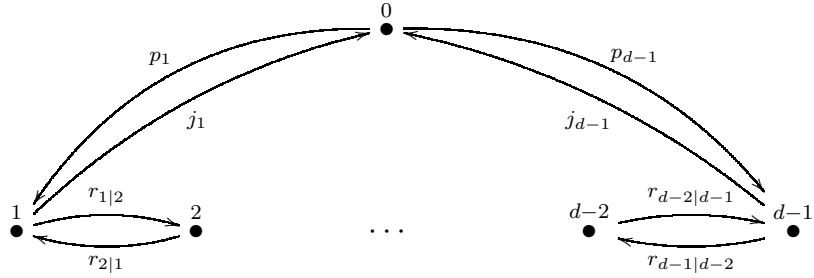
\begin{equation}
\xymatrix{
&&&&\overset{0}{\bullet}\ar@/_2pc/_{p_1}[ddllll]\ar@/^2pc/^{p_{d-1}}[ddrrrr]&&&&\\\\
\overset{1 }{\bullet}\ar@/^0.5pc/^{r_{1\vert 2}}[rr]\ar@/^1pc/_{j_1}[uurrrr]&& \overset{2}{\bullet}\ar@/^0.5pc/^{r_{2\vert 1}}[ll]&& \cdots && \overset{d-2}{\bullet} \ar@/^0.5pc/^{r_{d-2\vert d-1}}[rr]&& \overset{d-1}{\bullet}\ar@/^0.5pc/^{r_{d-1\vert d-2}}[ll]\ar@/_1pc/^{j_{d-1}}[uullll]
}
\end{equation}
modulo the relations defining $\widehat{Z}$ under the isomorphism sending $r_{i\vert i\pm 1}$ to $i\pm 1\vert i$, $p_i$ to $i\vert 0$ and $j_i$ to $0\vert i$ for $i\in \{1, d-1\}$. 
To verify the sign in the relation involving $0$ we observe that the endomorphism of $X_0$ given by $j_1p_1 + (-1)^{d-1}j_{d-1}p_{d-1}$ is (omitting shifts for readability) given by the solid arrows in the diagram
\begin{equation}
\xymatrix{Ze_{d-1}\ar[r]\ar_{\ell_{d-1}}[d]& Ze_{d-2}\ar^{0}[d]\ar^{r_{d-2\vert d-1}}@{-->}[dl] \ar[r]& \cdots\ar[r] &Ze_2\ar[r]\ar^{0}[d]\ar^{r_{2\vert 3}}@{-->}[dl] &Ze_1\ar^{\ell_1}[d]\ar^{r_{1\vert 2}}@{-->}[dl]\\
Ze_{d-1}\ar[r]& Ze_{d-2} \ar[r]& \cdots\ar[r] &Ze_2\ar[r] &Ze_1\\
}
\end{equation}
and is null-homotopic via the homotopy indicated by the dashed arrows.
\end{proof}

\begin{rem}
The natural bigrading of $\mathrm{End}_{ \mathbf{M}^{\Ev_{r,s}}}(X_0\oplus\cdots\oplus X_{d-1})$ induces a bigrading on $\widehat{Z}$ via the isomorphism in Proposition~\ref{prop:isoalgs}. Note that it does not depend on $(r,s)\in \mathbb{Z}^2$, 
as long as we keep the gradings of the $X_i$ fixed. Assuming that $(r,s)=(d-2,2-d)$, 
we see that it is given by 
\begin{eqnarray}
\deg(i\vert i+1) &= & (1,0), \quad  i\in I;\\
\deg(i\vert i-1) & = & (1,0), \quad i\in \widehat{I}\,\backslash \{1\};\\
\deg(0\vert 1)  & = &  (d+1,2-d);\\
\deg(1\vert 0) & = & (1-d, d-2).
\end{eqnarray} 
Note that the first entry of this bigrading is compatible with the above grading of $\widehat{Z}$ except for the degrees of the arrows between $0$ and $1$.
\end{rem}

The explicit $2$-action of $\Sext_d$ on $\widehat{\mathbf{M}}_{r,s}$ is given 
\begin{itemize}
\item on $1$-morphisms by
\begin{eqnarray}
F(i) & :=&\widehat{Z}_d e_i \otimes e_i\widehat{Z}_d\langle 1\rangle,\quad i \in \widehat{I};\\
F(\pm) & := & \widehat{Z}_d^{\tau^{\pm 1}}\langle r\rangle [s], 
\end{eqnarray}
\item on $2$-morphisms by
\begin{gather}
\begin{array}{lcrcl}
F\left(
\xy (0,0)*{
\tikzdiagc[scale=1]{
\begin{scope}[yscale=-.5,xscale=.5,shift={(5,-2)}] 
  \draw[ultra thick,blue] (-1,0) -- (-1, 1)node[pos=0, tikzdot]{} node[below]{\tiny $i$};
\end{scope}
}}
\endxy
\right)
& \colon & \widehat{Z}_d e_i \otimes e_i \widehat{Z}_d\langle 1 \rangle & \to  &\widehat{Z}_d \\
&& ae_i \otimes e_ib & \mapsto  & ae_ib,  
\end{array}\\
\begin{array}{lcrcl}
F\left(
\xy (0,0)*{
\tikzdiagc[scale=1]{
\begin{scope}[yscale=-.5,xscale=.5,shift={(5,-2)}] 
  \draw[ultra thick,blue] (-1,0) -- (-1, 1)node[pos=1, tikzdot]{} node[above, yshift=14pt]{\tiny $i$};
\end{scope}
}}
\endxy
\right)
& \colon &  \widehat{Z}_d  & \to  & \widehat{Z}_d e_i \otimes e_i \widehat{Z}_d\langle 1 \rangle \\
&& e_j & \mapsto  &
\begin{cases}
(-1)^i \left(\ell_i \otimes e_i + e_i \otimes \ell_i\right), & j=i;\\
(-1)^i\left(j\vert i \otimes i \vert j\right), & j\pm 1=i\ne 0;  \\
1\vert 0 \otimes 0 \vert 1, & j=1, i=0;\\
(-1)^d (d-1\vert 0 \otimes 0 \vert d-1), & j=d-1, i=0,
\end{cases}
\end{array}\\
\begin{array}{lcrcl}
F\left(\xy (0,0)*{
\tikzdiagc[scale=1]{
\begin{scope}[yscale=.5,xscale=.5,shift={(8,2)}] 
   \draw[ultra thick,blue] (0,0)-- (0, 1) node[above]{\tiny $i$}; 
  \draw[ultra thick,blue] (-1,-1) -- (0,0) node[below, shift={(-0.5,-0.5)}]{\tiny $i$}; 
\draw[ultra thick,blue] (1,-1) -- (0,0) node[below, shift={(0.5,-0.5)}]{\tiny $i$};
\end{scope}
}}
\endxy
\right)
&\colon & \widehat{Z}_d e_i \otimes e_i  \widehat{Z}_d e_i \otimes e_i  \widehat{Z}_d \langle 2 \rangle & \to 
&  \widehat{Z}_d e_i \otimes e_i  \widehat{Z}_d \langle 1 \rangle \\
&& e_i \otimes e_i ae_i \otimes e_i  &\mapsto & (-1)^i \mathrm{tr}(e_i a e_i) e_i \otimes e_i 
\\[2ex]
F\left(\xy (0,0)*{
\tikzdiagc[scale=1]{
\begin{scope}[yscale=.5,xscale=.5,shift={(8,2)}] 
   \draw[ultra thick,blue] (0,0)-- (0, 1) node[below, yshift=-14pt]{\tiny $i$}; 
  \draw[ultra thick,blue] (0,1) -- (-1,2) node[above]{\tiny $i$}; 
\draw[ultra thick,blue] (0,1) -- (1,2) node[above]{\tiny $i$};
\end{scope}
}}
\endxy
\right)
&\colon &  \widehat{Z}_d e_i \otimes e_i  \widehat{Z}_d \langle 1 \rangle 
&\to &  \widehat{Z}_d e_i \otimes e_i  \widehat{Z}_d e_i \otimes e_i  \widehat{Z}_d \langle 2 \rangle \\
&& e_i \otimes e_i &\mapsto & e_i \otimes e_i \otimes e_i.
\end{array}
\end{gather}
The generating $2$-morphisms involving an oriented black strand in 
\eqref{eq:orientedcoloredgens1} and \eqref{eq:orientedcoloredgens2} are sent to the 
isomorphisms in \eqref{eq:lefttau1} and \eqref{eq:lefttau3}, respectively,  
and all other generating $2$-morphisms are sent to zero.
\end{itemize}

\begin{rem}
We could alternatively have used the evaluation functor $\Ev'_{-r,-s}$ to obtain another evaluation birepresentation and its finitary cover.
\end{rem}





\vspace*{1cm}


\end{document}